\documentclass[11pt]{article}        

\usepackage{etoolbox}

\usepackage{amsthm,amssymb}

\usepackage{graphicx}

\usepackage{amscd,amssymb}
\usepackage{amsmath}
\usepackage{tikz}
\usepackage{tikz-3dplot}
\usepackage{tikz-cd}
\usepackage{faktor}

\newtheorem  {theorem}                  {Theorem}
\newtheorem* {theorem*}                   {Theorem}

\newtheorem {lemma}[theorem] {Lemma}
\newtheorem {prop}[theorem]      {Proposition}
\newtheorem* {prop*}     {Proposition}

\newtheorem {corollary}[theorem]      {Corollary}

\theoremstyle{definition}
\newtheorem {defi}[theorem] {Definition}
\newtheorem {Remark} [theorem]         {Remark}

\newtheorem* {Example*}    {Example}

\newcommand{\pp}[2]{\frac{\partial#1}{\partial#2}}
\newcommand{\norm}[1]{\|#1\|}

\title{Stability is not open or generic in symplectic four-manifolds}

\author{Robert Cardona\footnote{Robert Cardona acknowledges financial support from the Margarita Salas postdoctoral
contract financed by the European Union-NextGenerationEU, the Universitat Politècnica de Catalunya and the Instituto de Ciencias Matemáticas. This work was partially supported by the AEI grant PID2019-103849GB-I00 / AEI / 10.13039/501100011033, AGAUR grant 2021 SGR 00603 and the project Computational,
dynamical and geometrical complexity in fluid dynamics -  AYUDAS FUNDACIÓN
BBVA A PROYECTOS INVESTIGACIÓN CIENTÍFICA 2021.}}

\date{}

\begin{document}
\maketitle

\begin{abstract}
Given an embedded stable hypersurface in a four-dimensional symplectic manifold, we prove that it is stable isotopic to a $C^0$-close stable hypersurface with the following property: $C^\infty$-nearby hypersurfaces are generically unstable. This shows that the stability property is neither open nor generic, independently of the isotopy class of hypersurfaces and ambient symplectic manifold. The proof combines tools from stable Hamiltonian topology with techniques in three-dimensional dynamics such as partial sections, integrability and KAM theory. On our way, we establish non-density properties of Reeb-like flows and a generic non-integrability theorem for cohomologous Hamiltonian structures and volume-preserving fields in dimension three.
\end{abstract}

\tableofcontents

\section{Introduction}

The study of embedded hypersurfaces is a rich topic in symplectic geometry, further motivated by its connection to Hamiltonian dynamics. If $M$ denotes an embedded hypersurface in a symplectic manifold $(W,\Omega)$, the symplectic form restricted to $M$ defines a closed maximally non-degenerate two-form whose kernel integrates to what is known as the characteristic foliation of $M$. Any section of this line bundle defines a vector field that is parallel to the Hamiltonian vector field defined by any function having $M$ as a regular level set and which preserves some volume form in $M$. Among hypersurfaces in symplectic manifolds, those of contact type are the most studied in the literature. When $W$ is four-dimensional, a lot is known about the dynamics of their characteristic foliation, which is spanned by the Reeb vector field of a contact form. These hypersurfaces are a particular case of the more general class of stable hypersurfaces, introduced by Hofer and Zehnder in 1994 \cite{HZ}. A possible characterization of stable hypersurfaces is that they admit a distinguished symplectic tubular neighborhood foliated by hypersurfaces whose characteristic foliations are all isomorphic. If $\omega$ is the induced Hamiltonian structure (the pullback of the ambient symplectic form) on $M$, the hypersurface is stable if and only if the Hamiltonian structure is ``stabilizable". This means that there exists a one-form $\lambda$ such that the characteristic foliation evaluates positively on $\lambda$, and is contained in the kernel of $d\lambda$. The pair $(\lambda,\omega)$ is then called a stable Hamiltonian structure and it defines a Reeb vector field spanning the characteristic foliation of $M$. Notice that if $\alpha$ is a contact form, then $(\alpha,d\alpha)$ is a stable Hamiltonian structure, so these last provide a generalization of contact forms and their Reeb fields.

The existence of the distinguished tubular neighborhood makes stable hypersurfaces suitable as boundary conditions when studying pseudoholomorphic curves in symplectic cobordisms, as developed in symplectic field theory \cite{EGH, BEHWZ}. They appear in other works in symplectic topology, see for example \cite{EKP, NW}. From a dynamical perspective it can be shown that, except for some three-manifolds, the characteristic foliation always admits periodic orbits \cite{HT, Re, CR}.  Since their introduction, the dynamical properties of Reeb vector fields of stable Hamiltonian structures have been studied in different contexts, see \cite{CFP1, MP, CV, CV2, CR} and references therein. 

\subsection{Stability is not open or generic.}
An in-depth study of stable Hamiltonian structures in dimension three was undertaken by Cieliebak and Volkov in their seminal work \cite{CV, CV0}.  It is mentioned there that it is easy to show that stability is not a closed condition, but a harder question is whether stability is open, for instance in the $C^\infty$-topology, which is a very natural topological property. Indeed, some subclasses of stable hypersurfaces are defined by $C^1$-open conditions, such as contact hypersurfaces or those hypersurfaces whose characteristic foliation admits a global cross-section. When studying pseudoholomorphic curves in symplectic cobordisms, the usual requirement is that the boundary hypersurface is not only stable, but also non-degenerate. If stability was open, then any stable boundary hypersurface could be perturbed to be non-degenerate and stable simultaneously. It remains unknown whether a stable hypersurface can be approximated by non-degenerate stable hypersurfaces. From a dynamical point of view, the potential openness of the stability condition is meaningful since anything that can be said about their dynamics would apply to an open set of energy values of a Hamiltonian system with a stable energy level set. A weaker statement could hold, for instance, that hypersurfaces in a neighborhood of a stable hypersurface are generically stable. Unfortunately, it was shown by Cieliebak, Frauenfelder, and Paternain \cite{CFP} that there are at least some examples of six-dimensional symplectic manifolds with boundary that admit stable hypersurfaces whose stability is lost under a generic $C^\infty$-perturbation. This implies that stability is not, at least in full generality, an open or a generic property. Their example relies on the existence, in dimension five, of Hamiltonian structures whose kernel is spanned by an Anosov flow whose stability is lost under (generic) perturbations. This approach cannot be adapted to four dimensions: the only smooth Anosov flows in three dimensions that lie in the kernel of a stabilizable Hamiltonian structure are either suspensions or Reeb fields of contact forms. These two properties make any $C^1$-close Hamiltonian structure necessarily stable and hence any stable hypersurface with induced Anosov dynamics will have a $C^1$-neighborhood of stable hypersurfaces. 

In dimension four the openness question, often attributed to G. Paternain, had not been answered yet. The main contribution of this work shows that the stability condition is not open or generic (or generic in some open set of hypersurfaces) in four-dimensional symplectic manifolds. On a symplectic manifold $(W,\Omega)$, we denote by $\mathcal{HS}(W)$ the space of closed hypersurfaces equipped with the $C^\infty$-topology.

\begin{theorem}\label{thm:main}
Let $M$ be a closed embedded stable hypersurface in a symplectic four-manifold $(W,\Omega)$. Then $M$ is stable isotopic to an arbitrarily $C^0$-close stable hypersurface $\tilde M$ with the following property: it admits a $C^\infty$-neighborhood in $\mathcal{HS}(W)$ containing a residual set of unstable hypersurfaces.
\end{theorem}

Our theorem shows that even if we try to restrict to a particular symplectic manifold or a particular isotopy class of stable hypersurfaces (see Section \ref{ss:SHS} for a formal definition of stable isotopy, which can loosely be defined as an isotopy of stable hypersurfaces), one cannot achieve that stability is open or generic in some neighborhood of that isotopy class. Our statement applies indeed to the standard symplectic space $(\mathbb{R}^4,\omega_{std})$ for instance by choosing any convex hypersurface like the unit sphere, which is contact and hence stable. The methods in this work probably allow replacing $C^\infty$-generic by $C^k$-generic with $k\geq 5$ and considering hypersurfaces that are only $C^k$. This minimal regularity is needed though, to apply KAM arguments. Theorem \ref{thm:main} is sharp in the following sense: the fact that the isotopy is arbitrarily $C^0$-small cannot be replaced by $C^1$-smallness. If we start with a contact hypersurface then any $C^1$-close hypersurface is of contact type as well, so it admits a neighborhood in $\mathcal{HS}(W)$ without unstable hypersurfaces. This is also the case for hypersurfaces whose characteristic foliation admits global cross-section: this property implies stability and is robust under $C^1$-perturbations. It might be that these are the only two cases in which a stable hypersurface is stable after perturbation: there is no known example of robustly stable hypersurface that is not of contact type or with a global cross section.

 The steps of the proof of Theorem \ref{thm:main}, which we will summarize at the end of this introduction, yield statements of independent interest. Besides Theorem \ref{thm:Reeb} below, we show as well that on closed three-manifolds the kernel of a generic Hamiltonian structure admits no smooth first integral, see Corollary \ref{cor:exnonint}. This is proven in the language of vector fields preserving a volume form $\mu$ and whose dual two-form lies in a given cohomology class (Theorem \ref{thm:residualint}), and holds even if those flows admit zeroes. This is relevant to hydrodynamics and the helicity invariant \cite{EPST, KPSY}, as discussed after Corollary \ref{cor:exnonint}.

\subsection{Non-density of Reeb-like flows.} Another statement that we can establish when building towards Theorem \ref{thm:main} concerns Reeb-like flows on three-manifolds. Recently, several generic properties of Reeb fields defined by contact forms on three-manifolds have been established, such as the generic existence of dense and equidistributed periodic orbits \cite{Ir1, Ir2}, or the generic existence of Birkhoff sections by Colin-Dehornoy-Hryniewicz-Rechtman and Contreras-Mazzucchelli \cite{CDHR, CM}. These results involve important theories such as embedded contact homology \cite{H} and the existence of broken book decompositions  \cite{CDR}. The generic existence of Birkhoff sections and other results on the number of periodic orbits were generalized for Reeb fields defined by stable Hamiltonian structures in \cite{CR}, thus covering a larger set of non-vanishing volume-preserving vector fields. The following question then naturally arises: are flows conformal to a Reeb field (or ``Reeb-like flows") dense in the $C^k$-topology for some $k$ among either volume-preserving flows or flows conformal to the Reeb field of a stable Hamiltonian structure (or ``stable-like flows")? To state our result, let $M$ be closed three-manifold and denote by $\mathcal{R}(M)$ the set of Reeb-like flows, by $\mathcal{SR}(M)$ the set of flows conformal to Reeb fields of stable Hamiltonian structures, and by $\mathfrak{X}_{vol}(M)$ the set of volume-preserving flows in $M$. We denote as well by $\mathcal{R}_\mu(M)$, $\mathcal{SR}_\mu(M)$ and $\mathfrak{X}_{\mu}(M)$ the subsets of $\mathcal{R}(M)$, $\mathcal{SR}(M)$ and $\mathfrak{X}_{vol}(M)$ preserving a fixed volume form $\mu$ in $M$. In this language, Theorem \ref{thm:main} implies that $\mathcal{SR}_\mu(M)$ is not open in $\mathfrak{X}_{\mu}(M)$, see Corollary \ref{cor:notopen}.

In some concrete manifolds, it is easy to construct particular examples of vector fields in $\mathcal{SR}(M)$ which cannot be $C^0$-approximated by Reeb-like flows: for example volume-preserving vector fields admitting a global cross-section are stable-like, but cannot be Reeb-like. There exist as well, in some manifolds, volume-preserving Anosov flows that are not suspensions nor orbit equivalent to contact Anosov flows, like the class of non $\mathbb{R}$-covered Anosov flows (see \cite{BI} for a method to construct such examples). Those cannot be $C^1$-approximated by flows in $\mathcal{R}(M)$ because of their structural stability. The general case (including the three-sphere) is not covered by these examples, which only exist in restricted classes of three-manifolds. Our result establishes that this density does not occur on any manifold and homotopy class of non-vanishing vector fields.

\begin{theorem}\label{thm:Reeb}
Let $M$ be a closed three-manifold. In any homotopy class of non-vanishing vector fields there exists $Y\in \mathcal{SR}(M)$ and a $C^1$-neighborhood $\mathcal{V}$ of $Y$ in $\mathfrak{X}(M)$ such that $\mathcal{V}\cap \mathcal{R}(M)=\emptyset$.
\end{theorem}

In particular $\mathcal{R}(M)$ is not $C^1$-dense in $\mathfrak{X}_{vol}(M)$ or in $\mathcal{SR}(M)$. This implies that the inclusions $\mathcal{R}_\mu(M) \subset \mathcal{SR}_\mu(M)$ and $\mathcal{R}_\mu(M) \subset \mathfrak{X}_{\mu}(M)$ are not $C^1$-dense neither, see Corollary \ref{cor:fixmu}. We point out that this result is not implied, for example in the three-sphere, by the existence of Hamiltonian structures that cannot be approximated by contact type Hamiltonian structures, see e.g. \cite{Cie} (see Remark \ref{rem:ctype}). It is not known, to the author's knowledge, whether Reeb-like vector fields are $C^0$-dense among volume-preserving vector fields in the three-sphere. There are no known results in the three-sphere or that hold on every three-manifold about the density or non-density (in any topology) of $\mathcal{SR}_\mu(M)$ in $\mathfrak{X}_\mu(M)$.

\subsection{Outline of the proof.} Let us sketch the main ideas of the proof of Theorem \ref{thm:main}. The Reeb flow of a stable Hamiltonian structure $(\lambda,\omega)$ in dimension three necessarily satisfies at least one of the following three properties: it is Reeb-like, it admits a global cross-section or it admits a non-trivial first integral, see Section \ref{ss:SHS}. We will show that a given embedded stable hypersurface $\iota:M\rightarrow (W,\Omega)$ is stable isotopic to another one such that a generic perturbation yields a hypersurface whose characteristic foliation does not satisfy any of these three properties and hence it is not stable. 

For the moment, assume that we have constructed an embedded stable hypersurface $\tilde M=e(M)$ where
$$ e : M \longrightarrow (W,\Omega)$$
is $C^0$-close and stable isotopic to $\iota$, and that has the following property. A non-vanishing section (and hence any) of the characteristic line bundle of the induced Hamiltonian structure $\tilde \omega=e^*\Omega$ satisfies that any $C^\infty$-perturbation of that vector field is not Reeb-like and does not admit a global cross-section. Then any stable hypersurface in some $C^\infty$-neighborhood of $\tilde M$ necessarily admits a smooth first integral. If we fix a volume form $\mu$, the hypersurface $\tilde M$ defines a vector field preserving $\mu$ by considering $\iota_X\mu=\tilde \omega$. In fact, any nearby hypersurface can be identified with a vector field $Y$ close to $X$, preserving $\mu$, and such that $[\iota_Y\mu]=[\tilde \omega]$ in De Rham cohomology, see Appendix \ref{s:app}. We establish in Theorem \ref{thm:residualint} that a generic vector field preserving $\mu$ on a given cohomology class (understood as the class of $\iota_X\mu$) admits no first integral, using KAM theory as in the works of Markus and Meyer \cite{MM}. This shows that a generic perturbation of $\tilde M$ is unstable.\\

We have then reduced the proof to constructing the stable hypersurface $\tilde M$. First, we introduce the following obstruction to being Reeb-like: if a flow admits a positive and a negative partial section, then it is not Reeb-like. This can be used to give a refinement (Corollary \ref{cor:refTw}) of a previous obstruction introduced by Cieliebak and Volkov, identifying suitable partial sections on integrable regions of a flow. We show that in a region of $M$ where $d\lambda/\omega$ is non-constant, a region that we assume for the moment to exist, it is possible to $C^0$-deform the stable hypersurface such that its characteristic foliation admits a positive and a negative partial section, see Proposition \ref{prop:SHSht}. Finally, by a new deformation into a stable hypersurface $\tilde M$, we make these partial sections $\partial$-strong (see Section \ref{ss:deltastrong} for a definition of this property) and their binding orbits non-degenerate for the characteristic foliation of $\tilde M$, this is the content of Theorem \ref{thm:SHSdelta}. The existence of such a pair of sections becomes then a $C^1$-robust property of the vector field spanning the kernel of $\tilde \omega$. This construction can be used to prove, independently of our main result, the non $C^1$-density of Reeb-like flows, Theorem \ref{thm:Reeb}. The robust partial sections introduced in the characteristic foliation are shown to obstruct as well the existence of a global cross-section (Proposition \ref{prop:obsGCS}), thus the characteristic foliation of any hypersurface near $\tilde M$ is not Reeb-like and does not admit a global cross-section. By our previous considerations, this proves Theorem \ref{thm:main} under the simplifying assumption that $d\lambda/\omega$ is not constant (Theorem \ref{thm:mainsimp}).

We end up proving in Theorem \ref{thm:addintS} that given any stable hypersurface, it is isotopic to an arbitrarily $C^1$-close stable hypersurface that satisfies the simplifying assumption, thus proving the result in full generality. This step involves proving that a hypersurface in a symplectic manifold can be $C^1$-perturbed into one whose characteristic foliation admits an integrable region surrounding a given periodic orbit (see Proposition \ref{prop:addint}). Throughout the different arguments, we show that the isotopies can always be made stable. 

\paragraph{Organization of this paper.} In Section \ref{s:pre}, we introduce some preliminary notions on volume-preserving vector fields and stable Hamiltonian structures. In Section \ref{s:Reebobs}, we introduce some obstructions from partial sections for a flow to be Reeb-like or to admit a global cross-section. In Section \ref{s:addpartial}, we show that a stable hypersurface with a $T^2$-invariant integrable region is isotopic to a $C^0$-close hypersurface admitting a positive and a negative $\partial$-strong partial section with non-degenerate binding orbits. We deduce as well Theorem \ref{thm:Reeb}. In Section \ref{s:nonint}, we establish that generic volume-preserving vector fields on a given cohomology class admit no smooth first integral. Finally, in Section \ref{s:notgeneric}, we prove the main Theorem under the simplifying assumption and then reduce to this case by a preliminary perturbation of the stable hypersurface that introduces integrable regions on its characteristic foliation. The appendix contains a lemma relating close cohomologous Hamiltonian structures and embedded hypersurfaces in symplectic manifolds.

\paragraph{Acknowledgements.} The author is grateful to Daniel Peralta-Salas for enlightening discussions regarding KAM theory, and to the anonymous referees for their careful reading, positive feedback, and valuable suggestions which improved the quality of the paper.

\section{Preliminaries}\label{s:pre}

In this section, we recall some aspects of stable hypersurfaces and of volume-preserving flows on three-manifolds.

\subsection{Stable hypersurfaces and Reeb-like properties of flows}\label{ss:SHS}
We first recall some basic notions about stable Hamiltonian structures. A closed maximally non-degenerate two-form $\omega$ on a three-dimensional manifold $M$ is called a Hamiltonian structure. It defines a one-dimensional kernel that integrates to a one-dimensional foliation called the characteristic foliation of $\omega$. If $M$ is an embedded hypersurface in a symplectic manifold $(W,\Omega)$, the characteristic foliation of $M$ is the characteristic foliation of the Hamiltonian structure obtained by restriction of the ambient symplectic form.

A contact form on a three-dimensional manifold is a one-form $\alpha$ satisfying the non-integrability condition $\alpha\wedge d\alpha\neq 0$. The plane field distribution $\xi=\ker \alpha$ is a contact structure, and any positive multiple of $\alpha$ is another contact form defining $\xi$. Each contact form $\alpha$ uniquely determines a vector field $R$ (known as the Reeb field) by the equations
\begin{equation*}
\begin{cases}
\alpha(R)=1,\\
\iota_Rd\alpha=0.
\end{cases}
\end{equation*}
The Reeb field preserves the volume form $\alpha\wedge d\alpha$, and any positive multiple of $R$ is also volume-preserving since $fR$ preserves $\frac{1}{f}\alpha\wedge d\alpha$.  A non-vanishing vector field $X$ on a three-dimensional manifold $M$ that is a positive multiple of a Reeb field is called Reeb-like, and can be defined as follows.
\begin{defi}
A vector field $X$ is \textbf{Reeb-like} if there exists a contact form $\alpha$ such that $\alpha(X)>0$ and $\iota_Xd\alpha=0$.
\end{defi}
Given a three-manifold the set of Reeb-like fields $\mathcal{R}(M)$ defines a subset of the set $\mathfrak{X}_{vol}(M)$ of vector fields in $M$ preserving some volume form. Throughout the paper, we will often identify a vector field with its flow, thus speaking indistinguishably of a Reeb-like field, of a Reeb-like flow, or even of a Reeb-like characteristic foliation. \\ 

 A generalization of contact forms are stable Hamiltonian structures.

\begin{defi}
A stable Hamiltonian structure on a three-manifold $M$ is a pair $(\lambda,\omega)$ where $\lambda$ is a one-form and $\omega$ a two-form satisfying:
\begin{itemize}
\item[-] $\lambda\wedge \omega \neq 0$,
\item[-] $d\omega=0$
\item[-] $\ker \omega \subset \ker d\lambda$.
\end{itemize}
\end{defi}
\noindent A stable Hamiltonian structure uniquely determines a Reeb field as well, defined by the equations 
\begin{equation*}
 \begin{cases}  
 \lambda(X)=1,\\
 \iota_X\omega=0,
 \end{cases}
 \end{equation*}
 which spans the characteristic foliation of $\omega$. A vector field $X$ that is the positive multiple of a Reeb field defined by a stable Hamiltonian structure also preserves some volume form (possibly different from the one preserved by the Reeb field). We will call such vector field stable-like, a property that can be characterized as follows.
 \begin{defi}
 A vector field $X$ is \textbf{stable-like} if there exists a stable Hamiltonian structure $(\lambda,\omega)$ such that $\lambda(X)>0$ and $\iota_X\omega=0$.
 \end{defi}

The set of vector fields that are stable-like, or equivalently positive multiples of Reeb fields defined by stable Hamiltonian structures, defines a subset $\mathcal{SR}(M) \subset \mathfrak{X}_{vol}(M)$ which contains Reeb-like vector fields. 

\begin{defi}
A one-form $\lambda \in \Omega^1(M)$ stabilizes $\omega$ if $(\lambda,\omega)$ defines a stable Hamiltonian structure.
\end{defi}
A Hamiltonian structure does not necessarily admit a stabilizing one-form at all, and if it does it will not be unique in general. When it does, we say that $\omega$ is stabilizable. Observe that if $\alpha$ is a contact form then $(\alpha,d\alpha)$ is a stable Hamiltonian structure, so contact type Hamiltonian structures are a particular case of stabilizable Hamiltonian structures. 
\begin{defi} 
An embedded hypersurface $\iota:M\rightarrow W$ in a symplectic manifold $(W,\Omega)$ is stable if $\omega=\iota^*\Omega$ admits a stabilizing one-form. 
\end{defi}
Given any stabilizing one-form $\lambda$ of $\omega$, we will call the pair $(\lambda,\omega)$ an induced stable Hamiltonian structure on the stable hypersurface $M$.
This is equivalent to the definition introduced by Hofer and Zehnder \cite{HZ}, which states that $M$ admits a neighborhood $M\times (-\varepsilon,\varepsilon)$ in $W$ such that the characteristic foliation of each hypersurface $M\times \{c\}$ with $c\in (-\varepsilon,\varepsilon)$ is diffeomorphic to the characteristic foliation of $M \times \{0\}$. By our discussion above, contact hypersurfaces, for which $\iota^*\Omega=d\alpha$ for some contact form $\alpha$, are a particular class of stable hypersurfaces. 
Notice that if the characteristic foliation is stable-like then the hypersurface is stable, but an analogous property does not hold in the contact setting. A hypersurface is not necessarily of contact type if the characteristic foliation is Reeb-like, as the foliation might simultaneously be in the kernel of a contact type and a non contact type Hamiltonian structure.\\

Given two stable hypersurfaces on a symplectic manifold, a notion of isotopy between them was introduced in \cite[Section 6.6]{CV}.

\begin{defi}
A homotopy of embedded hypersurfaces $M_t$ in a symplectic manifold $(W,\Omega)$, with $t\in [0,1]$, is called stable if there exists a smooth homotopy of one-forms $\lambda_t \in \Omega^1(M)$ stabilizing ${\iota_t}^*\Omega$, where $\iota_t: M \hookrightarrow W$ is the embedding such that $\iota_t(M)=M_t$.
\end{defi}
We will say in this case that $M_1$ is stable isotopic to $M_0$. A potentially weaker notion, that we will not use here, would only require that $M_t$ is stable for each $t\in[0,1]$. We can say in this case that the homotopy $M_t$ is weakly stable. It could be the case, however, that every weak stable homotopy is stable. Let us mention that one can easily glue together stable homotopies as follows. If a stable homotopy ends at $(M_1,\lambda_1)$ and another one starts at $(M_1,\lambda_1')$ with $\lambda_1\neq \lambda_1'$, then one can interpolate linearly between $\lambda_1$ and $\lambda_1'$ while keeping the homotopy of hypersurfaces fixed at $M_1$. A linear combination of stabilizing forms is again stabilizing, so this allows us to compose stable homotopies of hypersurfaces.

\subsection{Integrability and $T^2$-invariant regions.} \label{ss:inteT2}

In this work, the notion of integrability will play an important role. Let $X$ be a non-vanishing volume-preserving vector field on a three-dimensional manifold $M$.

\begin{defi}\label{def:intreg}
A domain $U\subset M$ diffeomorphic to $T^2\times I$ is called an integrable region of $X$ if the vector field is tangent to the torus fibers in $U$.
\end{defi}
Sternberg's theorem \cite{Ste} shows that on each torus fiber, the flow is smoothly orbit equivalent (although not necessarily conjugate) to a linear flow. This implies the following result that we will use throughout this work.

\begin{theorem}[Sternberg]\label{thm:Ste}
Let $X$ be a volume-preserving vector field on a three-manifold $M$ and $U$ an integrable region of $X$. Then there exist coordinates $(x,y,r)$ of $U\cong T^2\times I$ such that 
\begin{equation}\label{eq:St}
    X|_U= G(x,y,r) \left( f_1(r) \pp{}{x} + f_2(r)\pp{}{y}\right),
\end{equation}
for some smooth functions $G \in C^\infty(U)$ and $f_1,f_2 \in C^\infty(I)$.
\end{theorem}

Following \cite[Section 3.2]{CV}, let us define the slope function of $X$ along $U$. We denote by $r$ the coordinate in the second factor of $T^2\times I$, and write the volume form $\mu$ preserved by $X$ as $\mu=dr\wedge \sigma_r$, where $\sigma_r$ is a smooth family of area forms in $T^2$ preserved by $X|_{T^2\times \{r\}}=X_r$. Looking at $X_r$ as a one-parameter family of vector fields in the torus, the family $\{\iota_{X_r}\sigma_r\}_{r\in I}$ is a family of non-vanishing closed one-forms in the torus. If we replace $X_r$ by another family of vector fields $\hat X_r$ defining the same oriented foliation as $X_r$ and preserving another family of area-forms $\hat \sigma_r$, then $[\iota_{\hat X_r}\hat \sigma_r]$ is a positive multiple of $[\iota_{X_r}\hat \sigma_r]$, see \cite[Lemma 3.5]{CV}. As the one-forms $\iota_{X_r}\sigma_r$ are non-vanishing, they define an element of the projectivization $PH^1(T^2,\mathbb{R})\cong S^1$ of the first de Rham cohomology group of the torus. The function that assigns to each $r\in I$ this element, understood as a point in the circle, is called the slope function
$$ k_U: I \rightarrow PH^1(T^2;\mathbb{R})\cong S^1. $$
As we have seen, the slope only depends on the oriented foliation defined by $X$, so one can speak of the slope of the foliation, the line bundle, or of a non-degenerate closed two-form having $X$ in its kernel. We will always work with the slope function in coordinates where the vector field $X$ is linear in each torus fiber, so let us give an expression of $k_U$ in that situation. If $(x, y,r)$ are coordinates in $T^2\times I$ and the vector field $X$ has the expression
$$X=f_1(r) \pp{}{x} + f_2(r)\pp{}{y}, $$
then 
$$k_U(r)=\frac{(f_2(r),-f_1(r))}{|(f_1(r),f_2(r))|},$$
and notice that this function is orthogonal to the slope vector of $X$ in the $(x,y)$ directions.\\

Let now $(\lambda,\omega)$ be a stable Hamiltonian structure on a three-manifold $M$, and denote by $X$ its Reeb field. The function $f=\frac{d\lambda}{\omega}$ is in this case a first integral of $X$. If $f$ is constant but non-zero, then the two-form $\omega$ is of contact type: it is exact and admits a primitive that is a contact form. If $f\equiv 0$, then $\lambda$ is closed and this implies by a classical argument of Tischler that $M$ fibers over the circle and $X$ admits a global cross-section. Recall that a global cross-section is an embedded closed surface in $M$ transverse to $X$ and intersecting every orbit of the flow.

 On any region $U\cong T^2\times I$ where the kernel of $\omega$ is tangent to the torus fibers, there exist \cite[Theorem 3.3]{CV} coordinates $x,y,r$ such that
\begin{equation}\label{eq:wT2}
\omega= h_1'(r)dr\wedge dx + h_2'(r) dr\wedge dy,
\end{equation}
where $h_1,h_2 \in C^\infty(I)$ and 
$$\lambda= g_1(r)dx + g_2(r)dy + g_3(r,x,y)dr,$$
for some functions $g_1,g_2\in C^\infty(I)$ and $g_3\in C^\infty(U)$. We write $\omega$ in terms of the derivatives of $h_1,h_2$ because in this region $\omega$ is necessarily exact, and admits a primitive $\omega=d\alpha$ of the form $\alpha= h_1(r)dx + h_2(r)dy$. Notice that the kernel of $\omega$ is generated by $X=-h_2'(r)\pp{}{x} + h_1'(r) \pp{}{y}$, so the slope function can be expressed as $k_U(r)=\frac{(h_1'(r),h_2'(r))}{|(h_1',h_2')|}$.

When $\omega$ is written as in Equation \eqref{eq:wT2}, we will say, abusing terminology, that $\omega$ is a $T^2$-invariant Hamiltonian structure in $U$. If furthermore the function $f$ is constant on each torus, which is the case if the integrable region is defined by this first integral, then the expression of $\lambda$ is $\lambda= g_1(r)dx + g_2(r)dy+ g_3(r)dr.$
Finally, up to a homotopy of stabilizing one-forms \cite[Lemma 3.9]{CV} relative to the boundary of $U$, we can assume that $\lambda$ takes the form
\begin{equation}\label{eq:lam}
\lambda= g_1(r)dx + g_2(r)dy,
\end{equation}
in a smaller integrable region contained in $U$. We will always consider this simplification. Again abusing terminology, we say that $\lambda$ is $T^2$-invariant if it can be written as in Equation \eqref{eq:lam}. If both $\lambda$ and $\omega$ are $T^2$-invariant in $U\cong T^2\times I$, we say that $U$ is a $T^2$-invariant region of the stable Hamiltonian structure. It can be helpful to understand the $T^2$-invariant forms $\omega$ and $\lambda$ as curves $H=(h_1,h_2):I\rightarrow \mathbb{R}^2$ and $G=(g_1,g_2):I\rightarrow \mathbb{R}^2$ in the plane. For instance, in this language, the primitive one-form $\alpha$ is a contact form if and only if $H$ is an immersion that turns all the time in the same direction, i.e. clockwise or counterclockwise with respect to the origin. If we are given a $T^2$-invariant one-form that stabilizes a $T^2$-invariant Hamiltonian structure near the boundary of $T^2\times I$, the following Proposition (a restatement of \cite[Proposition 3.14]{CV}) allows us to extend it to all $T^2\times I$ as long as the kernel of the Hamiltonian structure has non-constant slope.

\begin{prop}\label{prop:CV1}
Let $\omega$ be a $T^2$-invariant Hamiltonian structure in $T^2\times I$, and assume that the slope of $\omega$ is non-constant. Given $T^2$-invariant stabilizing forms $\lambda_0,\lambda_1$ of $\omega$ near $r=0$ and $r=1$ respectively, there exists a $T^2$-invariant stabilizing form $\lambda$ of $\omega$ such that $\lambda=\lambda_0$ near $r=0$ and $\lambda=\lambda_1$ near $r=1$.
\end{prop}
This result will be used several times to construct stabilizing one-forms. When one starts with a $T^2$-invariant stabilizing one-form and we deform the Hamiltonian structure in a $T^2$-invariant way and relative to the boundary, we need a parametric variant of this proposition. The following statement combines \cite[Lemma 3.10]{CV} and \cite[Proposition 3.17]{CV}.

\begin{prop}\label{prop:param}
    Let $(\lambda_0,\omega_0)$ be a $T^2$-invariant stable Hamiltonian structure in $T^2\times I$. Let $\omega_t$ be a homotopy of $T^2$-invariant Hamiltonian structures such that $\omega_t=\omega_0$ in $T^2\times ([0,\delta]\cup [1-\delta,1])$. Assume either that the slope of $\omega_0$ is constant in $[0,\delta]\cup [1-\delta, 1]$ or that the slope of $\omega_t$ is non-constant in $[\delta,1-\delta]$ for each $t$. Then there exists a homotopy of $T^2$-invariant stable Hamiltonian structures $(\lambda_t,\omega_t)$ such that $(\lambda_t,\omega_t)=(\lambda_0,\omega_0)$ in $T^2\times ([0,\epsilon]\cup [1-\epsilon,1])$, where $\varepsilon<\delta$.
\end{prop}

\section{Obstructions from partial sections}\label{s:Reebobs}

In this section, we discuss some obstructions for a volume-preserving vector field $X$ on a three-dimensional manifold $M$ to be Reeb-like or to admit a global cross-section. Both obstructions will come from the existence of certain partial sections (see Definition \ref{defi:partialsec} below). A simple example of an obstruction for a flow to be Reeb-like is the existence of an embedded closed surface transverse to the flow: this prevents $X$ from being Reeb-like by a Stokes argument. This is satisfied for flows admitting a global cross-section, or for flows transverse to taut foliations with a compact leaf as constructed in \cite{Ga}. Other simple obstructions include, for example, the existence of an invariant Reeb cylinder, an object that also prevents the flow from being stable-like \cite{PRT}.

\subsection{Signed partial sections and Reeb-like obstructions}\label{ss:partial}

The main obstructions that we use in this work come from the theory of partial sections.

\begin{defi}\label{defi:partialsec}
Let $X$ be a non-vanishing vector field on a three-manifold $M$. An immersed surface with boundary $\iota: \Sigma \rightarrow M$ is a partial section of $X$ if
\begin{itemize}
\item[-] the interior of $\iota(\Sigma)$ is embedded and transverse to $X$,
\item[-] the boundary of $\iota(\Sigma)$ is immersed and tangent to $X$.
\end{itemize}
\end{defi}
Abusing notation, we will sometimes use $\Sigma$ to denote as well the immersed surface in $M$. We will call the periodic orbits of $X$ that are in the image of $\iota|_{\partial \Sigma}$ the \emph{binding} orbits of $\Sigma$, in analogy with open book decompositions. A partial section can be oriented by the unique orientation that makes $X$ positively transverse to its interior. This induces an orientation on its boundary components, that are mapped to periodic orbits of $X$. We say that the immersed surface $\Sigma$ is a \emph{positive} partial section if the orientation induced on each boundary component of $\Sigma$ corresponds to the orientation induced by the positive direction of $X$. If the orientation on each boundary component is opposite to the one induced by the flow, we call $\Sigma$ a \emph{negative} partial section. Note that in general, a partial section is neither positive nor negative. Our first use of signed partial sections is that they serve as obstructions to the Reeb-like property.

\begin{lemma}\label{lem:obs}
Let $X$ be a vector field in $M$ that admits a positive and a negative partial section. Then $X$ is not Reeb-like.
\end{lemma}
\begin{proof}
Let $\Sigma_1$ and $\Sigma_2$ be a positive and a negative partial section of $X$ respectively, oriented by the unique orientation that makes $X$ positively transverse to them. The boundaries $\partial \Sigma_1=\bigcup_{i=1}^k \sigma_i^1$ and $\partial \Sigma_2= \bigcup_{j=1}^r \sigma_j^2$ are mapped in $M$ to two collections of embedded closed orbits of $X$. By assumption, each curve $\sigma_i^1$ is oriented positively with respect to $X$, and each curve $\sigma_j^2$ is oriented with the opposite orientation with respect to $X$.  By contradiction, assume that there exists a contact form $\alpha$ such that $\alpha(X)>0$ and $\iota_Xd\alpha=0$.  Having fixed an orientation of $M$, if $\alpha\wedge d\alpha>0$ then $d\alpha$ is a positive area-form in the interior of $\Sigma_1$ and $\Sigma_2$. In particular $\int_{\Sigma_2} d\alpha>0$, but by Stokes theorem
$$\int_{\Sigma_2}d\alpha= \sum_{j=1}^r \int_{\sigma_j^2} \alpha<0,$$
reaching a contradiction. Analogously if $\alpha\wedge d\alpha<0$, then $d\alpha$ is a negative area-form in $\Sigma_1$ and $\Sigma_2$. We reach a contradiction since $\int_{\Sigma_1}d\alpha<0$ but by Stokes theorem
$$\int_{\Sigma_1}d\alpha=\sum_{i=1}^k \int_{\sigma_i^1} \alpha>0.$$
 We conclude that $X$ is not Reeb-like.
\end{proof}
The existence of a partial section is, in general, not robust under perturbations. Under technical conditions it can become a robust property (see Section \ref{ss:deltastrong}).
\begin{Remark}
In \cite[Theorem 1.4]{PRT}, a characterization of Reeb-like flows was given using geometric currents, in a similar fashion of McDuff-Sullivan's characterization of contact type two-forms \cite{Mc, Su}. Assuming that $M$ does not fiber over the circle, a vector field $X$ in $M$ is Reeb-like if and only if for some volume form $\mu$ there is no sequence of zero-flux two-chains whose boundary approximates a foliation cycle in the weak topology. A two-chain $c:\Omega^2(M)\rightarrow \mathbb{R}$ is called zero-flux if $c(\iota_X\mu)=0$. We can illustrate this theorem in Lemma \ref{lem:obs}. Orient the positive and the negative partial sections such that their boundary are foliation cycles (i.e. the boundaries are oriented positively with respect to the flow). This implies that whatever the volume form $\mu$ is, the flux through one of the sections is positive, and the flux through the other one is negative. Hence, a positive linear combination of both sections (this combination depends on $\mu$) defines a zero-flux two-chain whose boundary is a foliation cycle, showing that $X$ is not Reeb-like.
\end{Remark}

\begin{Remark}\label{rem:ctype}
A property related to Reeb-like flows is the contact type property of two-forms. An exact non-degenerate two-form is of contact type if it admits a primitive contact form. The kernel of such a two-form is always a Reeb-like flow. McDuff-Sullivan's characterization \cite{Mc,Su} of the contact type property can be used to define obstructions for a two-form to be of contact type, see for instance \cite{Cie}. We point out, however, that these obstructions do not hold as obstructions for the kernel of the two-form to be Reeb-like, like Lemma \ref{lem:obs} does. Indeed, when the line bundle defined by $\ker d\alpha$ admits a non-trivial smooth first integral, there are examples where the line bundle lies in the kernel of several non-degenerate exact two-forms, where some are of contact type and other are not.
\end{Remark}

\subsection{Partial sections on integrable regions}

 From the previous obstructions we deduce a refinement of a result obtained by Cieliebak and Volkov. Let $X$ be a vector field in $M$ admitting an integrable region $U\cong T^2\times I$. The following definition was introduced in \cite{CV}. We say that the slope function $k_U$ of $X$ in $U$ \emph{twists} if it makes one full turn clockwise and one full turn counterclockwise. The following statement was proven in \cite[Theorem 3.36]{CV}.

\begin{theorem}[\cite{CV}]\label{thm:CVtw}
If a volume-preserving vector field $X$ admits an integrable region where the slope function twists, then $X$ is not Reeb-like.
\end{theorem}

It turns out that the previous theorem can be deduced using Lemma \ref{lem:obs}, and we obtain a refinement as well. In the next definition, we orient $S^1$ in the clockwise direction, so that given two points $p,q \in S^1$, the interval $[p,q]$ denotes the interval starting at $p$ and ending at $q$ rotating in the clockwise direction. Similarly, the interval $[q,p]$ denotes the interval starting at $p$ and ending at $q$ rotating in the counterclockwise direction.

\begin{defi}
A vector field $X$ admits a positive (respectively negative) half-twisting integrable region if there exists an integrable region $U\cong T^2\times I$ where $X$ is tangent to the torus fibers, and such that the slope $k_U: I\rightarrow S^1$ of $X$ in $U$ satisfies the following property: there exists an interval $[a,b] \subset I$ such that $k(b)=-k(a)$ and $k(r)\subset (k(a),k(b))$ (respectively $k(r)\in (k(b),k(a))$) for each $r\in (a,b)$. If the endpoints $k(a), k(b)$ are rational points in $S^1$, we will say that $X$ admits a rational positive (respectively negative) half-twisting region.
\end{defi}
 Of course, the fact that an integrable region is (positive or negative) half-twisting does not depend on the chosen coordinates in $T^2\times I$. The orientation of $M$ and a choice of coordinate $t$ in $I$ induces an orientation of $T^2$ and hence of $S^1 \cong PH^1(T^2;\mathbb{R})$, so that the sign of a half-twisting region is well-defined and does not depend on the chosen coordinates either.

If $X$ admits an integrable region where the slope twists, then it necessarily admits a positive half-twisting region and a negative half-twisting region. The property of having a positive and a negative half-twisting region is, however, less restrictive and not local: it might happen that each half-twisting region lies in a different integrable region (i.e. the two regions are not contained in a larger connected integrable region).

\begin{lemma}\label{lem:htw}
Let $X$ be a volume-preserving vector field in $M$. If the vector field $X$ admits a positive (respectively negative) rational half-twisting region, then $X$ admits a positive (respectively negative) partial section.
\end{lemma}
\begin{proof}
We will treat the positive case, the negative one being analogous. By hypothesis, there is an integrable region $U\cong T^2\times I$ where the slope function of $X$ positively half-twists. By Theorem \ref{thm:Ste} there are coordinates $(x,y,r)$ of $U$ where the flow is, perhaps after a reparametrization, of the form
$$ X= f_1(r) \pp{}{x} + f_2(r) \pp{}{y}. $$
In these coordinates, the slope function takes the form
$$ k_1:= \frac{(f_2(r),-f_1(r))}{|(f_1(r),f_2(r))|}\longrightarrow S^1.  $$
By continuity, there is some interval $[a,b]\subset [0,1]$ satisfying $k_1(a)=-k_1(b)$ and $k_1(t)\in (k_1(a),k_1(b))$ for each $r\in (a,b)$. Furthermore, the vector $k_1(a)=(p,q)$ is rational. Let $\gamma$ be a closed periodic orbit of $X$ in $T^2\times \{a\}$, and consider the embedded surface with boundary
$$ \Sigma= \gamma \times [a,b] \subset T^2\times I. $$
The vector field $X$ is transverse to the interior of $\Sigma$, and we orient it again by the orientation making $X$ positively transverse to $\Sigma$. 

\begin{figure}[!ht]
\begin{center}
\begin{tikzpicture}
     \node[anchor=south west,inner sep=0] at (0,0) {\includegraphics[scale=0.08]{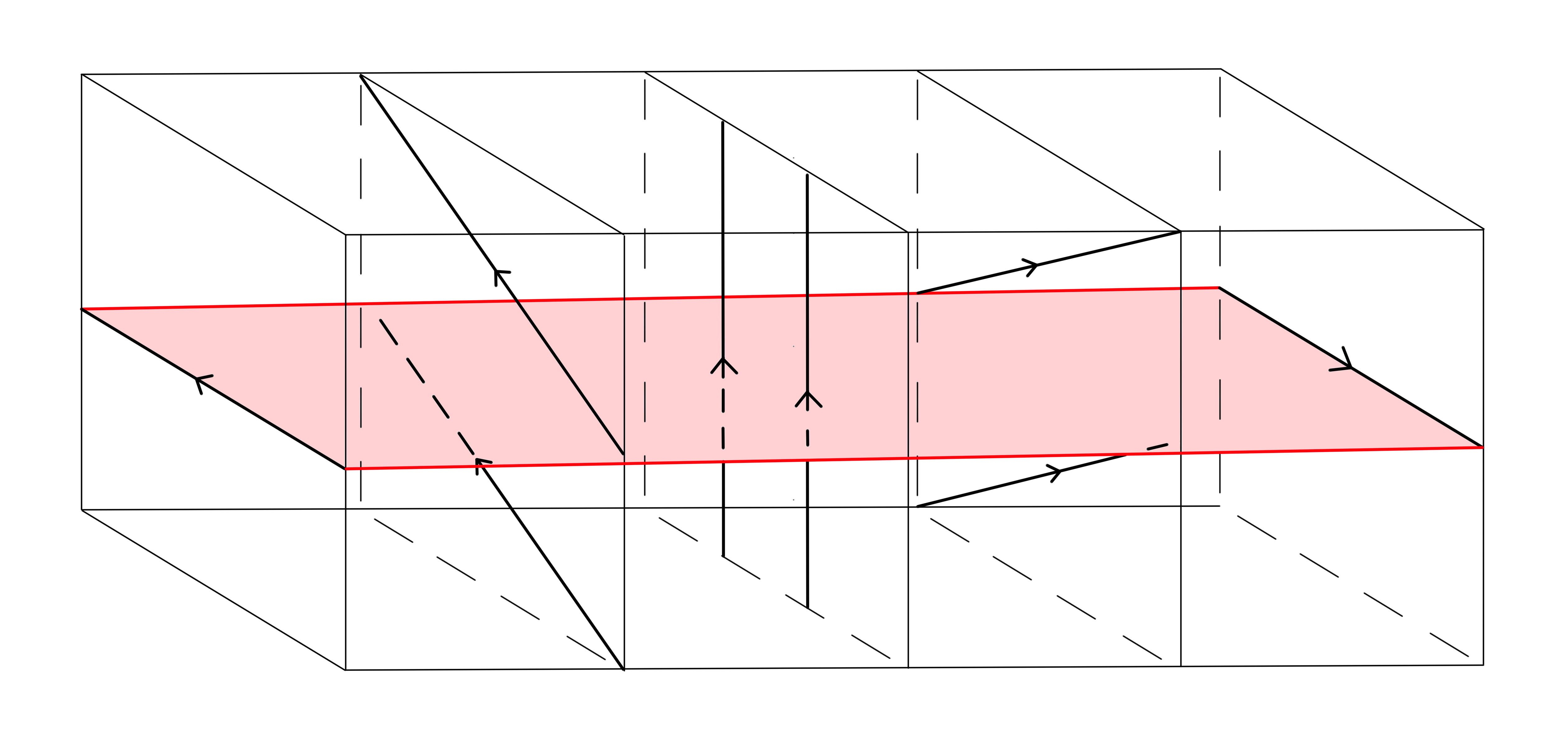}};
     \draw[->] (-0.2+11.2-11.5,1.4-0.5)--(11.2+0.4-11.5,1.4-0.5);
     \draw[->] (-0.2+11.2-11.5,1.4-0.5)--(-0.2+11.2-11.5,1.9-0.5);
     \draw[->] (-0.2+11.2-11.5,1.4-0.5)--(+0.23+11.2-11.5,1.4-0.23-0.5);
     
     \node[scale=0.8] at ( 11.4-11.5,1.65-0.5) {$r$};
     \node[scale=0.8] at (11.6-11.5, 1.1-0.5) {$x$};
     \node[scale=0.8] at (11-11.5, 2.1-0.5) {$y$};
     \end{tikzpicture}
\caption{A positive partial section $\Sigma$ in a positive rational half-twisting region}
\label{fig:part}
\end{center}
\end{figure}

Using that the half-twisting is positive, we deduce that the induced orientation in the boundary orbits $\gamma \times \{a,b\}$ corresponds to the positive orientation of the flow $X$, i.e. $\Sigma$ is a positive partial section of $X$. See Figure \ref{fig:part} for a representation of the positive partial section.
\end{proof}

Combining Lemma \ref{lem:htw} and Lemma \ref{lem:obs}, we deduce a refinement of Theorem \ref{thm:CVtw}.

\begin{corollary}\label{cor:refTw}
If a volume-preserving vector field $X$ admits a positive and a negative rational half-twisting region, then $X$ is not Reeb-like.
\end{corollary}

Throughout this work, we will use the obstructions coming from partial sections and their relation to integrability. These obstructions will be instrumental in proving that the whole set of Reeb-like flows on a three-manifold is not $C^1$-dense among volume-preserving or stable-like flows (Theorem \ref{thm:Reeb}). Furthermore, we are able to introduce them in the characteristic foliation of a stable hypersurface in a symplectic manifold by $C^0$-perturbations, which will be an important step to prove Theorem \ref{thm:main}. 

\begin{Remark}
After this work was written, it was pointed out to the author that Corollary \ref{cor:refTw} and its proof using partial sections were independently observed by Becker, and can be extracted from the arguments in the proof of \cite[Proposition 7]{Be}.
\end{Remark}

\subsection{Obstructions to global cross-sections}\label{ss:obsGCS}

We have seen that the existence of a positive and a negative partial section is an obstruction for a vector field to be Reeb-like. It turns out that the existence of a single signed partial section prevents the vector field from admitting a global cross-section, a fact that we will also need in the proof of Theorem \ref{thm:main}. 

\begin{prop}\label{prop:obsGCS}
    Let $X$ be a non-vanishing vector field on a closed three-manifold $M$. If $X$ admits a signed partial section, then $X$ admits no global cross-section.
\end{prop}
\noindent To establish this result, we need to introduce a fundamental result of Schwartzman's theory of cross-sections \cite{Sc} that uses the language of foliation cycles \cite{Su}. 

\begin{defi}
    A $k$-current is a continuous linear functional defined on $\Omega^k(M)$. We denote by $\mathcal{Z}^k(M)$ the space of $k$-currents in $M$.
\end{defi} 
\noindent There exists a continuous boundary operator
\begin{align*}
\partial: \mathcal{Z}^k(M) &\longrightarrow \mathcal{Z}^{k-1}(M)\\
			c &\longmapsto \partial c,
\end{align*}
 that sends a $k$-current $c$ to the $(k-1)$-current defined as
$$\partial c(\beta):=c(d\beta), \enspace \text{for any } \beta \in \Omega^{k-1}(M).$$ 
A $k$-current $c$ is closed if $\partial c=0$, and exact if there exists a $(k+1)$-current $c'$ such that $\partial c'=c$. The boundary operator satisfies $\partial^2\equiv 0$, so that every exact current is closed. Given a non-vanishing vector field $X$ we define the Dirac $1$-current $\delta_p\in \mathcal{Z}^1(M)$ as
$$\delta_p(\beta)=\beta(X)|_p, \enspace \beta \in \Omega^1(M).$$ 
The closed convex cone generated by Dirac $1$-currents defines the set of foliation currents of $X$. A foliation cycle is simply a closed foliation current. 

Recall as well that an embedded closed surface 
$$\iota: \Sigma \longrightarrow M $$
is a global cross-section of $X$ if $\Sigma \pitchfork X$ and $\Sigma$ intersects every orbit. Schwartzman characterized those flows that admit a global cross-section in terms of their exact foliation cycles.

\begin{theorem}[Schwartzman \cite{Sc}]\label{thm:Sch}
Let $X$ be a non-vanishing vector field on a closed manifold $M$. Then $X$ admits a global cross-section if and only if $X$ does not admit any exact foliation cycle.
\end{theorem}
With this theorem we are ready to prove Proposition \ref{prop:obsGCS}.

\begin{proof}[Proof of Proposition \ref{prop:obsGCS}]
Let $X$ be a non-vanishing vector field in $M$, and assume that there is a signed partial section 
$$\iota:\Sigma \longrightarrow M,$$
which we orient by the only orientation that makes $X$ positively transverse to its interior.
This partial section defines a $2$-current $c_{\Sigma}$ by integration $c_{\Sigma}(\eta)=\int_\Sigma \eta$, where $\eta \in \Omega^2(M)$. We define $c$ as $c_{\Sigma}$ if $\Sigma$ is a positive partial section, or as $-c_{\Sigma}$ if $\Sigma$ is a negative partial section. Recall that the boundary of $\Sigma$ is a finite cover of a collection of periodic orbits $\gamma_1,...,\gamma_k$ oriented by the direction of $X$ (or by the reversed orientation if $\Sigma$ is a negative partial section). The current $\partial c$ is defined by integration as well, acting on a one-form $\beta$ by
$$\partial c (\beta)= \sum_{i=1}^k k_i \int_{\gamma_i} \beta, \quad k_i\in \mathbb{Z}^+.$$
By construction, the current $\partial c$ is an exact cycle, and notice that the fact that $\gamma_i$ is positively tangent to $X$ implies that $\partial c$ is a foliation current of $X$. By Theorem \ref{thm:Sch} the flow $X$ does not admit a global cross-section.
\end{proof}

\section{Partial sections in stable hypersurfaces}\label{s:addpartial}

The goal of this section is to show how to deform a stable hypersurface into one admitting positive and negative partial sections that persist under perturbations. We first show how to introduce positive and negative rational half-twisting regions in a $T^2$-invariant region $U\cong T^2\times I$, thus creating positive and negative partial sections by Lemma \ref{lem:htw}. We further perturb the hypersurface to make these sections $\partial$-strong and the binding orbits non-degenerate, making them robust.

\subsection{Half-twisting regions}
To introduce rational half-twisting regions in the characteristic foliation of a stable hypersurface, we use the fact that a region where $(\lambda,\omega)$ is $T^2$-invariant admits a distinguished symplectization, which corresponds to a neighborhood of that region in the ambient symplectic manifold $(W,\Omega)$. This symplectization is well-suited for $T^2$-invariant embeddings, see an example of these embeddings in \cite[Section 3.11] {CV}.

\begin{prop}\label{prop:SHSht}
Let $M \subset (W,\Omega)$ be a closed stable hypersurface in a four-dimensional symplectic manifold and $\omega$ its induced Hamiltonian structure. Assume that $\omega$ admits a stabilizing one-form $\lambda$ such that $d\lambda=f\omega$ with $f$ non-constant. Then there exists an arbitrarily $C^0$-close embedded stable hypersurface, stable isotopic to $M$, that admits a positive and a negative rational half-twisting integrable region where an associated stable Hamiltonian structure is $T^2$-invariant. 
\end{prop}
\begin{proof}
We will first show the existence of a $C^0$-close isotopic stable hypersurface satisfying the conclusions, and argue at the end that the isotopy can be chosen to be stable. Since $f$ is non-constant, there exists a domain $U\cong T^2\times I$ where $(\lambda,\omega)$ is $T^2$-invariant, see Section \ref{ss:inteT2}. Namely, there exist coordinates $(x,y,r)$ of $U$ where $\lambda$ takes the form
$$ \lambda= g_1(r)dx + g_2(r)dy$$
and $\omega$ has the expression
$$\omega=h_1'(r)dr\wedge dx + h_2'(r) dr\wedge dy,$$ 
where $g_1,g_2,h_1,h_2$ are smooth functions defined from $[0,1]$ to $\mathbb{R}$. Furthermore, we can choose $U$ such that $f|_U\neq 0$ and hence that $d\lambda$ is non-vanishing everywhere. Up to taking a small interval inside $[0,1]$ and restricting to it, we can also assume that the slope function of the Reeb field of $(\lambda,\omega)$ has its image in a small neighborhood of a point in $S^1$. If possible, we fix some $\delta$ for which the slope of $\omega$ is constant in $[0,\delta]\cup [1-\delta,1]$. Otherwise, and up to shrinking $U$, we fix $\delta$ such that the slope is not constant in $[1-2\delta, 1-\delta]$. Either of these conditions will be sufficient to apply Proposition \ref{prop:param} later.\\

We claim that, up to homotoping $\lambda$ to another stabilizing one-form of $\omega$, we can assume that the vector $g(r)=(g_1(r),g_2(r))\in \mathbb{R}^2$ is rational for some $r=a\in (\delta, 1-2\delta)$. If the slope of $g(r)$ is constant, then $g(r)=c(r)g_0$ for some non-zero vector $g_0=(a,b)$ and some positive function $c(r)$. Consider a non-zero function $\rho:I\rightarrow \mathbb{R}$, vanishing for $r\in [0,\delta]\cup [1-2\delta,1]$, and satisfying $\int_0^1 \rho(r)h_1'(r)dr=\int_0^1 \rho(r) h_2'(r)dr=0$. Such a function exists since it is the zero level set of the linear functional $\mathcal{F}:C_c^\infty([0,1])\rightarrow \mathbb{R}^2$ given by $\mathcal{F}(f)=(\int_0^1 f(r)h_1'(r)dr,\int_0^1 f(r)h_2'(r)dr)$ is necessarily infinite-dimensional. Let $\tilde g_1(r)$ and $\tilde g_2(r)$ be primitives of $\rho h_1'$ and $\rho h_2'$ respectively, both vanishing for $r=0$, and define 
$$\tilde \lambda= \lambda + \tilde g_1 dx + \tilde g_2 dy. $$
Taking $\rho$ small enough, it is clear that $\tilde \lambda\wedge \omega>0$. In addition, we have $d\tilde \lambda= d\lambda + \rho \omega$, and hence $\tilde \lambda$ is a stabilizing one-form of $\omega$ (which we can connect to $\lambda$ by a linear homotopy of stabilizing one-forms). We claim that $g+\tilde g$ has non-constant slope, where $\tilde g=(\tilde g_1,\tilde g_2)$. The derivative of $\langle g+ \tilde g, g_0^{\perp}\rangle$ with respect to $r$ is $\rho(r) \langle h'(r) , g_0^\perp \rangle$, where $g_0^{\perp}=(-b,a)$ and $h'(r)=(h_1'(r),h_2'(r))$. But $\lambda\wedge \omega>0$ implies that $\langle h', g\rangle \neq 0$, and thus $\langle g+ \tilde g, g_0^{\perp}\rangle$ is non-constant. Hence $g+\tilde g$ is somewhere non-parallel to $g_0$, which shows that its slope is non-constant and in particular it is rational somewhere in $(\delta,1-2\delta)$ as we wanted. To avoid overloading the notation, we keep the notation $\lambda=g_1(r)dx+g_2(r)dy$ for the stabilizing one-form for which $(g_1(a),g_2(a))$ is rational.

Consider a neighborhood of $U$ in $(W,\Omega)$ identified with $V=U\times (-\varepsilon,\varepsilon)$ equipped with the symplectic form
\begin{align*}
\Omega_U&= \omega + d(t\lambda)\\
	   &= (h_1'+tg_1')dr\wedge dx + (h_2'+tg_2')dr\wedge dy + g_1 dt\wedge dx + g_2 dt\wedge y,
\end{align*}
where $t$ is the coordinate in $(-\varepsilon,\varepsilon)$. In this neighborhood we have $M|_V=U\times \{0\}$, i.e. $M|_V$ can be seen as the embedding $\iota: U \longrightarrow U\times \{0\}\subset V$ and the Hamiltonian structure $\omega$ is given by the pullback form $i^*\Omega_U$. Consider an embedding of the form
\begin{align}\label{eq:emb}
\begin{split}
F: T^2\times I \longrightarrow T^2\times I \times (-\varepsilon,\varepsilon)\\
	(x,y,s) \longmapsto (x,y, f_1(s), f_2(s)),
\end{split}
\end{align}
where $f=(f_1,f_2):I \rightarrow I\times (-\varepsilon,\varepsilon)$ is an embedding equal to $(s,0)$ near $s=0$ and near $s=1$, and such that $f((0,1))\subset (0,1)\times (-\varepsilon,\varepsilon)$. 

We choose the curve $f$ to be positively tangent to the vertical line $r=a$ for some $s_1\in I$, negatively tangent to the vertical line $r=a$ for some $s_2\in I$, positively tangent again for some $s_3$ and finally negatively tangent for some $s_4$ such that $0<s_1<s_2<s_3<s_4<1-2\delta$. We can impose as well $f_1'(s_i)=0$, $f_2'(s_1)=f_2'(s_3)=1$ and $f_2'(s_2)=f_2'(s_4)=-1$. See Figure \ref{fig:curv} for an example of such a curve. Notice that $f$ can be chosen arbitrarily $C^0$-small.

\begin{figure}[!ht]
\begin{center}
\begin{tikzpicture}
     \node[anchor=south west,inner sep=0] at (0,0) {\includegraphics[scale=0.4]{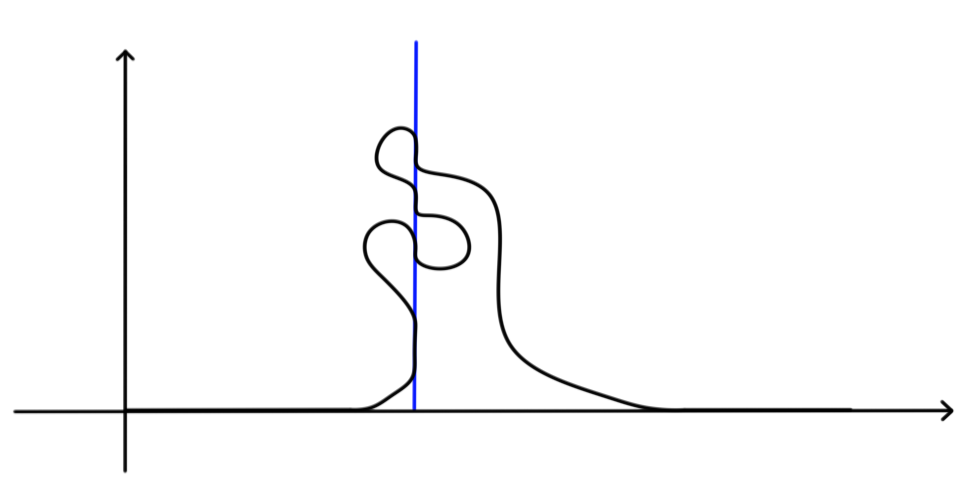}};
     \node[color=blue, scale=0.9] at (4.4,0.4) {$r=a$};
     \node[scale=1] at (10,0.4) {$r$};
     \node[scale=1] at (1,4.5) {$t$};
     
     \end{tikzpicture}
\caption{The curve $f:I\longrightarrow I\times [0,\varepsilon)$.}
\label{fig:curv}
\end{center}
\end{figure}

The Hamiltonian structure induced by the embedding $F$ is given by
\begin{align*}
\omega_F=F^*\Omega_U &= f_1'(s)\Big( (h_1'(f_1)+f_2 g_1'(f_1)) ds\wedge dx + (h_2'(f_1)+f_2 g_2'(f_1)) ds\wedge dy\Big) \\
&+ f_2'(s) \Big( g_1(f_1) ds\wedge dx + g_2(f_1) ds\wedge dy \Big)\\
&= H_1(s) ds\wedge dx + H_2(s) ds \wedge dy,
\end{align*}
where 
$$\begin{cases}
    H_1(s)= f_1'(s)h_1'(f_1(s))+f_1'(s)f_2(s)g_1'(f_1(s))+f_2'(s)g_1(f_1(s)),\\
    H_2(s)=f_1'(s)h_2'(f_1(s))+f_1'(s)f_2(s)g_2'(f_1(s))+f_2'(s)g_2(f_1(s)).
\end{cases}.$$
The slope function of $\omega_F$ is
$$k_U(s)=\frac{(H_2(s),-H_1(s))}{|(H_1,H_2)|},$$
and by construction
$$k_U(s_1)=k_U(s_3)=\frac{(g_2(a),-g_1(a))}{|(g_2(a),g_1(a))}, \quad k_U(s_2)=k_U(s_4)=\frac{(-g_2(a),g_1(a))}{|(-g_2(a),g_1(a))|}.$$ This shows that the slope of the kernel of $\omega_F$ visits respectively a rational point $q=k_U(s_1)$ in $S^1$, its image by the antipodal map $-q$, then $q$ and $-q$ again as $s$ runs from $0$ to $1$. Let us show that this implies that $\omega_F$ has a positive and a negative half-twisting region in $T^2\times I$.

The immersed curve $f$ is homotopic through immersed curves to the initial immersion $G=(s,0)$ relative to the boundary, so the slope function of the kernel of $\omega_F$ is homotopic relative to the endpoints to the slope function $k_0(s)$ of the Reeb field of $(\lambda,\omega)$, whose image is contained in a very small interval inside $S^1$. Take the covering map $\pi: \mathbb{R}\rightarrow S^1$ with $\pi(\rho)=e^{i\rho}$, and let $\tilde k$ and $\tilde k_0$ be respectively the lifts of $k_U$ and $k_0$. The fact that $k_U$ and $k_0$ are homotopic implies that $\tilde k(1)-\tilde k(0)=\tilde k_0(1)-\tilde k_0(0)=\theta_0$, where by assumption $|\theta_0|<<1$.

As $s$ runs through from $s_1$ to $s_2$, the image of the slope function goes from $q$ to $-q$, and thus the image of $k_U$ contains one of the half-circles between $q$ and $-q$. We assume it to be $(q,-q)$, i.e. the interval that goes from $q$ to $-q$ in the clockwise direction. The other case can be treated by analogous arguments to the ones we will use.  This implies that there is necessarily an interval $[a_1, a_2]\subset [s_1,s_2]$ such that $k_U(a_1)=q, k_U(a_2)=-q$ and $k_U(s)\in (q,-q)$, i.e. $T^2\times [a_1,a_2]$ is a rational positive half-twisting region of $\omega_F$. For instance choose $a_2=s_2$ and choose $a_1$ to be the largest value in $[s_1,s_2]$ such that $k_U(a_1)=-q$. 

We will assume now that there is no negative rational half-twisting region of $\omega_F$ in $T^2\times I$ and reach a contradiction. The fact that there is no negative rational half-twisting region implies that for any two values $b_1,b_2\in I$ with $b_1 \leq b_2$ and such that $k_U(b_1)$ or $k_U(b_2)$ is rational, we have 
\begin{equation}\label{eq:half}
    \tilde k(b_2)- \tilde k(b_1)> -\pi.
\end{equation}
 As $k_U(s_1)$ and $k_U(s_2)$ are rational and antipodal, we have $\tilde k(s_2)\geq \tilde k(s_1)+\pi$. Analogously, we have $\tilde k(s_3)\geq \tilde k(s_2)+\pi$ and $\tilde k(s_4)\geq \tilde k(s_3)+\pi$. Hence $\tilde k(s_4)\geq k(s_1)+3\pi$. Using that $\tilde k(1)-\tilde k(s_4)> -\pi$ and that $\tilde k(s_1)-\tilde k(0)> -\pi$, we conclude that $\tilde k(1)-\tilde k(0)> \pi$, which is a contradiction with the fact that $\tilde k(1)-\tilde k(0)$ should be equal to $\theta_0$ with $|\theta_0|<<1$. We conclude that there is a negative rational half-twisting region in $T^2\times I$.\\

We define our new embedded hypersurface $M_F$ as a hypersurface equal to $M$ away from $V$, and equal to $F(T^2 \times I)$ in $V$. Recall that $f$ was arbitrarily $C^0$-small, and thus $M_F$ can be chosen arbitrarily $C^0$-close to $M$. It remains to show that $M_F$ is stable. To see this, notice that the Hamiltonian structure $\omega_F$ coincides with $\omega$ away from the integrable region, hence $\lambda$ is a stabilizing one-form of $\omega_F$ there. In $T^2\times I$, the form $\lambda$ is a stabilizing one-form of $\omega_F$ near the boundary, since there $\omega_F=\omega$ as well. Furthermore, the slope of the kernel of $\omega_F$ is not constant by construction, so by Proposition \ref{prop:CV1} there exists an extension of $\lambda$ in $T^2\times I$ as a stabilizing one-form of $\omega_F$. This shows the existence of the required stable hypersurface.\\

Let us check that $M$ and $M_F$ are stable isotopic. Let $f_\tau$ be a one-parameter family of injective immersions such that $f_0=(s,0)$ and $f_1=f$. They define a family of embeddings $F_\tau$ as in Equation \eqref{eq:emb}. Each $F_\tau$ defines an embedded hypersurface $M_\tau$ which is equal to $M$ except in some region $V= U \times (-\varepsilon,\varepsilon)$, and has an induced $\omega_{F_\tau}$ in $U$. We can choose $F_\tau$ such that $F_\tau(s)=(s,0)$ for $s\in [0,\delta] \cup [1-2\delta,1]$. In particular, either $F_\tau$ has constant slope in $[0,\delta]\cup [1-\delta,1]$ for each $\tau$ or has non-constant slope in $[1-2\delta,1-\delta]\subset [\delta,1-\delta]$ for each $\tau$. We can apply Proposition \ref{prop:param} to obtain a parametric family of stabilizing one-forms $\lambda_\tau$ of $\omega_{F_\tau}$ that extends away from $U$ as a stabilizing one-form of $M_\tau$.
\end{proof}

\subsection{Robust partial sections}\label{ss:deltastrong}
We will now introduce an important feature of a partial section which is known as $\partial$-strongness, and which is needed to construct partial sections that persist (up to small isotopy) for perturbations of the flow, as we prove in Proposition \ref{prop:robpartial} below. We will then prove that robust partial sections can be introduced in the characteristic foliation of stable hypersurfaces by suitable perturbations.\\

Let $\iota:\Sigma \rightarrow M$ be a partial section of a vector field $X$. Let $\gamma$ be a binding orbit of $\Sigma$, and denote by $M_\gamma$ the manifold obtained by replacing $\gamma$ by its unit normal bundle.  We denote the quotient map by $\pi_\gamma:M_\gamma \rightarrow M$. Concretely ${T}_\gamma=\left((TM|_{\gamma}/T\gamma)\setminus 0\right)/ \mathbb{R}^+$, where we identify two vectors $v,w\in T_pM\setminus \{0\}$, where $p$ is a point in $\gamma$, if there exists a positive real number $\lambda$ such that $v-\lambda w\in T_p\gamma$. It defines a circle bundle 
$$p_\gamma:{T}_\gamma \longrightarrow \gamma,$$
and is diffeomorphic to a torus. The image of a boundary component $\sigma$ of $\Sigma$ defines an oriented immersed curve $\widetilde \sigma$ in $T_\gamma$, defined as follows (see e.g. \cite[Section 2.2]{CDHR}). If $\iota(\sigma)=\gamma$, choose any vector field $n$ tangent to $\Sigma$ along $\sigma$ pointing outwards, so that $d\iota(n)$ defines a map from $\sigma$ to $TM|_{\gamma}$. Composing $d\iota(n)$ with the quotient map $\pi_\gamma$ we obtain a map $\tau:\sigma \rightarrow {T}_\gamma$. This map continuously assigns to a point in $\gamma$ a section of the circle bundle $p_\gamma$, defining an immersed curve $\tilde \sigma$.  The multiplicity of $\sigma$ can be defined either by the degree of the map $\iota|_\sigma:\sigma \rightarrow \gamma$, where $\gamma$ is oriented by the flow, or by the degree of the map $\pi_\gamma|_{\widetilde \sigma}:\widetilde\sigma \rightarrow \gamma$, see \cite[Section 2.1]{ABM}. The curve $\tilde \sigma$ is always embedded if the multiplicity is one in absolute value, but could be only immersed if the multiplicity of $\sigma$ is greater than one in absolute value. This happens when at two different points $x,y\in \sigma$ such that $\iota(x)=\iota(y)$, we have $d\iota(n)|_x=d\iota(n)|_y$.  The boundary component $\sigma$ is positive or negative if the multiplicity is positive or negative respectively. 

More generally, if $L$ denotes the collection of periodic orbits of $X$ in the image of $\partial \Sigma$, we can define the blow-up manifold $M_L$ obtained by replacing each periodic orbit with its unit normal bundle. The linearized flow of $X$ determines a smooth flow $\tilde X$ in the boundary of $M_L$. Each orbit $\gamma\in L$ defines a boundary component ${T}_\gamma$ of $M_L$ defined as the unit normal bundle of $\gamma$. Each boundary component $\sigma$ of $\Sigma$ defines an immersed curve $\sigma^*$ in the boundary of $M_L$.
\begin{defi}
The partial section $\Sigma$ is $\partial$-strong if the set of curves $\bigcup_{\sigma \in \partial \Sigma} \sigma^*$ is embedded and transverse to the linearized flow in the boundary of $M_L$.
\end{defi}
Another property is needed for a partial section to be robust: we need that the binding orbits are non-degenerate. Recall that a periodic orbit is non-degenerate if the linearized Poincar\'e return map does not have any root of unity as an eigenvalue. The persistence of a partial section is formalized in the following proposition.

\begin{prop}\label{prop:robpartial}
Let $X$ be a vector field in $M$ that admits a $\partial$-strong partial section $\Sigma$ whose binding orbits are non-degenerate. Then there exists a $C^1$-neighborhood of $X$ in $\mathfrak{X}(M)$ such that every vector field in this neighborhood admits a $\partial$-strong partial section that is isotopic to $\Sigma$ by a $C^1$-small isotopy. The orientation of each of its boundary components is preserved.
\end{prop}
\begin{proof}
The proof follows from Lemma \ref{lem:obs} and the fact that $\partial$-strong partial sections with non-degenerate binding orbits are robust to $C^1$-perturbations, induced orientations included, which follows exactly from the arguments in \cite[Proposition 5.1]{CDHR} where the case of a Birkhoff section is treated.
\end{proof}

In the previous section, we showed how to construct stable hypersurfaces whose Reeb field admits a positive and a negative rational half-twisting regions where $(\lambda,\omega)$ is $T^2$-invariant. By Lemma \ref{lem:htw}, the Reeb field admits a positive and a negative partial section. The goal of this section is to show that one can improve this into a stable hypersurface whose Reeb field admits a positive and a negative partial section that are $\partial$-strong and whose binding orbits are non-degenerate.

\begin{theorem}\label{thm:SHSdelta}
Let $M$ be a stable hypersurface in $(W, \Omega)$. Let $(\lambda,\omega)$ be a stable Hamiltonian structure induced in $M$. Assume that $(\lambda,\omega)$ admits a rational positive half-twisting $T^2$-invariant region and a rational negatively half-twisting $T^2$-invariant region. Then there exists a stable hypersurface $\tilde M$ that is stable isotopic to $M$, arbitrarily $C^\infty$-close to $M$ and whose characteristic foliation admits a positive and a negative $\partial$-strong partial section with non-degenerate binding periodic orbits.
\end{theorem}

\begin{proof}
We will show that there exists a $C^\infty$-close stabilizable Hamiltonian structure $\hat \omega$ that is arbitrarily $C^\infty$-close to $\omega$ and satisfies $[\hat \omega]=[\omega]$ in $H^2(M,\mathbb{R})$ and whose kernel satisfies the desired properties. If this is proven, then by Lemma \ref{lem:app} in the appendix, the Hamiltonian structure $\hat \omega$ is given by the restriction of an embedding of $M$ in $(W,\Omega)$ that is $C^\infty$-close and isotopic to the original embedding. We will justify at the end of the proof that this isotopy is stable.\\

Let us work with the positive rational half-twisting region, and the negative one is treated analogously. By assumption there is an integrable region $U=T^2\times I$ with coordinates $(\theta,\varphi,r)$ where $\omega$ is of the form
$$\omega=d\alpha= h_1'(r)dr\wedge d\theta +h_2'(r)dr\wedge d\varphi, $$
and $\lambda$ is of the form
$$ \lambda=g_1(r)d\theta+ g_2(r)d\varphi. $$
We also know that the slope function, which in coordinates can be written as
$$k=\frac{(h_1'(r),h_2'(r))}{\norm{(h_1'(r),h_2'(r)}}:I\longrightarrow S^1,$$
half-twists positively. To simplify, assume that we did choose $\theta,\varphi$ such that for some small $\delta>0$ we have
\begin{equation}\label{eq:rot}
    (h_1'(\delta),h_2'(\delta))=(0,-1), \quad (h_1'(1-\delta),h_2'(1-\delta))=(0,1),
\end{equation} and $k(r)$ is contained between $(-1,0)$ and $(1,0)$ in the left hemisphere of $S^1$ for $r\in(\delta,1-\delta)$. \\

Our first goal is to show that there exists another stabilizing one-form $\tilde \lambda$ of $\omega$ which is closed near $r=\delta$ and $r=1-\delta$. This will be a useful property to make sure that a modification of $\omega$ that we will do later remains stable.

Choose a small enough $\tau>0$ and consider the intervals $I_1=[\tau, \delta-\tau]$, $I_2=[\delta+\tau, 1-\delta-\tau]$ and $I_3=[1-\delta+\tau,1-\tau]$. To ease the notation, we will assume that there is some $\tau$ such that the slope of the Reeb field of $(\lambda,\omega)$ is not constant in $I_1,I_2,I_3$. This is trivially true if we introduce several parameters $\tau_i$, but we will stick to a single $\tau$. In a small open neighborhood of $[0,\tau]$ in $I$, the one-form $\eta_0=g_1(r)d\theta+g_2(r)d\phi=\lambda|_{[0,\tau]}$ stabilizes $\omega$. In a neighborhood of $[\delta-\tau,\delta+\tau]$, we choose as one-form $\eta_1=d\theta$. The kernel of $\omega$ in $\{r=\delta\}$ is proportional to $-\pp{}{\theta}$, hence $\eta_1$ stabilizes $\omega$ in a neighborhood of $[\delta-\tau,\delta+\tau]$ for $\tau$ small enough. In a neighborhood of $[1-\delta-\tau, 1-\delta+\tau]$, we choose the one-form $\eta_2= -d\theta$, which by the same reason as before stabilizes $\omega$. In an open neighborhood of $[1-\tau,\tau]$ we choose $\eta_3=g_1(r)d\theta+g_2(r)d\phi=\lambda|_{[1-\tau,1]}$. We can now apply Proposition \ref{prop:CV1} in each interval $I_1,I_2,I_3$. We obtain a stabilizing one-form $\tilde \lambda$ of $\omega$ such that 
\begin{enumerate}
\item $\tilde \lambda=\lambda$ near $r=0$ and $r=1$,
\item $\tilde \lambda=d\theta$ near $r=\delta$,
\item $\tilde \lambda=-d\theta$ near $r=1-\delta$.
\end{enumerate}
Recall that the kernel of $\omega$ admits a positive partial section (Lemma \ref{lem:htw}) given by 
$$ \Sigma=\{(r,\theta,\varphi)\mid r\in[\delta,1-\delta], \varphi=\pi, \}. $$
A small neighborhood of the binding orbit $\gamma_1=\{r=\delta, \varphi=\pi\}$ is given by $U=\{(r,\theta,\varphi)\mid r\in(\delta-\tau_1,\delta+\tau_1), \varphi\in(\pi-\tau_1,\pi+\tau_1)\}$, which is isomorphic to $(-\tau_1,\tau_1)\times (-\tau_1,\tau_1)\times S^1$. In this last description let $(x,y)$ be coordinates in $(-\tau_1,\tau_1)$, i.e. $x=r-\delta$ and $y=\varphi-\pi$ in $U$. In these coordinates, we have
$$\omega=h_1'(x+\delta)dx\wedge d\theta + h_2'(x+\delta) dx\wedge dy, $$
and we know that $\frac{h_1'(x)}{h_2'(x)}$ is decreasing in $(0,\tau_1)$, by our assumptions on the slope function of the kernel of $\omega$. The stabilizing form in these coordinates is still written $\tilde \lambda=d\theta$. The partial section corresponds in $U$ to $\Sigma \cap U=\{x\geq 0, y=0\}$, see Figure \ref{fig:partdelt}). Take polar coordinates $(\rho,\phi)$ in a small enough two-disk $D\subset [-\tau_1,\tau_1]^2$ centered at zero. Then $\Sigma$ is just $\{\phi=0\}$, the two-form $\omega$ has the expression
$$\omega= h_1'(\rho\cos\phi+\delta)\cdot \left(\cos \phi d\rho \wedge d\theta-\sin \phi \rho d\phi\wedge d\theta\right)+h_2'(\rho\cos\phi+\delta) \rho d\rho \wedge d\phi.$$
and the kernel $\omega$ is generated by
$$X=h_1'(\rho\cos\phi+\delta)\sin\phi \pp{}{\rho} + h_1'(\rho\cos\phi+\delta)\frac{\cos \phi}{\rho}\pp{}{\phi} - h_2'(\rho\cos\phi+\delta)\pp{}{\theta}. $$
Now we modify $\omega$ by an exact perturbation
\begin{align*}
\tilde \omega&= \omega -\rho F(\rho)d\rho\wedge d\theta= \omega - d\big(\left( \int_0^\rho u F(u)du\right) d\theta \big).
\end{align*}
where $F(\rho)$ is a non-positive cut-off function equal to some negative constant $c>0$ near $\rho=0$, and equal to $0$ near $\rho=\delta'$ for some small $\delta'>0$. The one-form $\beta=\left( \int_0^\rho u F(u)du -  \int_0^{\delta'}u F(u) du\right) d\theta$ is also a primitive of the perturbation that we apply to $\omega$, and vanishes near the boundary of $U$. It is hence a globally defined one-form. If we can choose $c$ arbitrarily small, then $\beta$ can be assumed to be arbitrarily $C^\infty$-small. The kernel of $\tilde \omega$ is spanned by
$$ \tilde X= X + F(\rho)\pp{}{\phi}. $$
First, observe that the partial section $\Sigma$ is still a partial section of $X$. Indeed as $(h_1',h_2')$ is rotating clockwise near $r=\delta$ and $(h_1'(\delta),h_2'(\delta))=(0,-1)$, we have that $h_1'(r)<0$ for $r\in (\delta, \delta+ \delta_1)$ for a small $\delta_1$. Hence the vector field $F(\rho)\pp{}{\phi}$, which is transverse to $\Sigma$ for $x>0$, induces the same orientation as $X$ on $\Sigma$. The conclusion is that $\tilde X$ is positively transverse to the interior of $\Sigma$, see Figure \ref{fig:partdelt}.

\begin{figure}[!ht]
\begin{center}
\begin{tikzpicture}
     \node[anchor=south west,inner sep=0] at (0,0) {\includegraphics[scale=0.1]{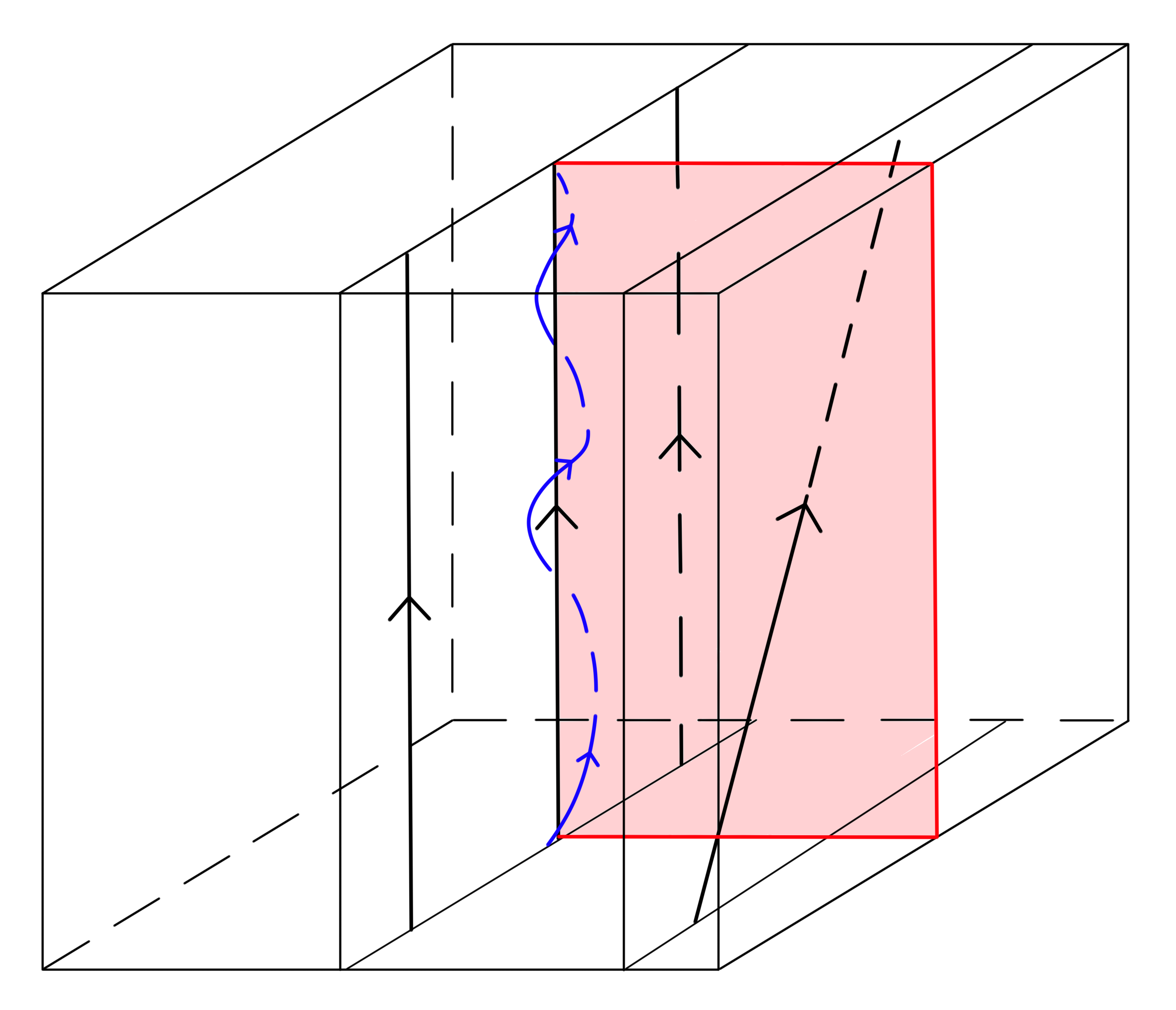}};
     \draw[->] (7,0.95)--(7,1.55);
     \draw[->] (7,0.95)--(7.6,0.95);
     \draw[->] (7,0.95)--(6.65,0.65);
     \node[scale=0.8] at (7,1.8) {$\theta$};
     \node[scale=0.8] at (7.8,0.95) {$x$};
     \node[scale=0.8] at (6.9,0.55) {$y$};
     \node[scale=0.8, color=red] at (5.7,4) {$\Sigma\cap U$};
     \node[scale=0.75, color=blue] at (2.98,5.05) {$ F(\rho)\pp{}{\phi}$};
     \end{tikzpicture}
\caption{The partial section and the perturbation of $X$}
\label{fig:partdelt}
\end{center}
\end{figure}

 Secondly, we claim that the form $\tilde \lambda$ is a stabilizing form of $\tilde \omega$. This has to be checked only in $D\times S^1$, since away from it the pair $(\tilde \lambda, \tilde \omega)$ coincides with $(\lambda,\omega)$. In $D\times S^1$, $\tilde \lambda$ is closed so the only condition to check is that $\tilde \lambda\wedge \tilde \omega>0$. But we know that in $\tilde \lambda$ in $U$ has the form $\tilde \lambda|_U=d\theta$, so for any choice of $F(\rho)$ we have $\lambda|_U(\tilde X)=\lambda_U(X)>0$. In fact, the form $\tilde \lambda$ is a stabilizing one-form for any Hamiltonian structure obtained by a compactly supported small perturbation of $\omega$ in $D\times S^1$, this will be useful to argue that the isotopy can be done through stable hypersurfaces.\\

Let us show that $\Sigma$ is a $\partial$-strong partial section of the Reeb field of $(\tilde \lambda,\tilde \omega)$. Let $M_L$ be the blow-up manifold obtained by replacing each periodic orbit in $\partial \Sigma$ by its unit normal bundle. In terms of our coordinates, the boundary of $M_L$ associated with $\gamma$ admits a collar neighborhood $\tilde U\cong T^2\times [0,1]$ parametrized by $(\theta,\phi, \rho)$ (we abuse notation by denoting the extended coordinates like the original ones), and the extended section $\Sigma_L$ corresponds to $\{\phi=0\}$. To compute the extended vector field $X_L$ we proceed as in \cite{Hr}. First, in $U$ the vector field $\tilde X$ has the form
$$\tilde X=h_1'\sin\phi \pp{}{\rho} +\left(F(\rho) + h_1'\frac{\cos \phi}{\rho}\right)\pp{}{\phi} - h_2'\pp{}{\theta}. $$
Second, recall that $h_2'|_{\rho=0}=-1$, so we have $X|_{(\theta,0,0)}=(1,0,0)$. We will now describe, following \cite[Section 3.1]{Hr}, the behavior of the extended vector field $\tilde X_L$ near the boundary component of $M_L$ associated with the closed orbit $\gamma$. It has the form
$$ \tilde X_L |_{\mathbb{T}_\gamma}= \pp{}{\theta} + b(\theta,\phi) \pp{}{\phi}, $$
 and the function $b$ can be described in terms of a decomposition of the form $\tilde X=u_1\pp{}{\theta}+u_2\pp{}{x}+u_3\pp{}{y}$ as follows. Let $x+iy$ be a complex coordinate of the disk $D$, then in our notation $x+iy=\rho e^{i\phi}$. Then
$$ b(\theta,\varphi)= \langle D_{2}Y(\theta,0)e^{i\phi}, ie^{i\phi} \rangle$$
where $Y=u_2\pp{}{x}+u_3\pp{}{y}$ and $D_2$ is the derivative in the disk factor, see \cite[Formula (14)]{Hr}. In our situation, the vector field $\tilde X$ decomposes as
$$ \tilde X= h_1'(x+\delta)\pp{}{y} - h_2'(x+\delta)\pp{}{\theta} + F(\rho) \left( -y \pp{}{x} + x \pp{}{y} \right). $$
A computation of $D_2Y$ near a small neighborhood of the origin, taking into account that $F(\rho)\equiv c$ for $\rho$ small enough, yields
\begin{equation*}
D_2Y(\theta,\rho,\phi)= 
\begin{pmatrix}
0 & -c\\
h_1''(x+ \delta)+c&0
\end{pmatrix} \implies D_2Y(\theta, 0, \phi)= 
\begin{pmatrix}
0 & -c\\
h_1''(\delta) +c&0
\end{pmatrix} 
\end{equation*}
We deduce that $b(\theta,\phi)=c\sin^2\phi+(h_1''(\delta)+c) \cos^2\phi$. Since $h'(r)$ is rotating positively, i.e. $(h_1'(r),h_2'(r))$ turns clockwise at $r=\delta$, and by Equation \eqref{eq:rot} we have $(h_1'(\delta),h_2'(\delta))=(0,-1)$, we deduce that $h_1''(\delta)\leq  0$. Choosing any negative value $c$ is enough, and then $\tilde X_L$ is everywhere transverse to the extended section $\Sigma_L$, which is given by the equation $\{\phi=0\}$. In particular, the constant $c$ can be arbitrarily small in absolute value and the one-form $\beta$ is $C^\infty$-small.

\begin{Remark} An alternative approach to perturb $X_L$ into some $\tilde X_L$ transverse to $\Sigma_L$ and such that the section becomes $\partial$-strong is by doing a $T^2$-invariant perturbation of the stable Hamiltonian structure in a way that $h_1''(\delta)>0$ near the tori containing the binding orbits. It might be possible to obtain this condition directly via an embedding as in Proposition \ref{prop:SHSht}, but it certainly does not follow immediately from the construction. We use a more general approach instead, which can be useful in other situations where one needs to produce a $\partial$-strong partial section out of one that is not, by slightly perturbing the vector field near the binding orbit. For the last step of this proof (making the binding orbits non-degenerate), it is key in any case that we change the stabilizing form near the binding orbit to be closed.
\end{Remark}

Doing this at every boundary component of $\Sigma$, and at each boundary component of the negative partial section in the negative rational half-twisting region, we end up with a stable Hamiltonian structure $(\tilde \lambda, \tilde \omega)$ whose Reeb field admits a positive and a negative $\partial$-strong partial section. Furthermore the Hamiltonian structure $\tilde \omega$ is cohomologous to $\omega$ and arbitrarily $C^\infty$-close to it. We do a final modification to $\tilde \omega$ to ensure that the binding orbits of the partial sections are non-degenerate. Near a binding orbit, using the previous notation, the stabilizing one-form is $\tilde \lambda=d\theta$ where $X|_{\gamma}=\pp{}{\theta}$. By \cite[Lemma 19]{Rob} there exists an arbitrarily $C^\infty$-small one-form $\eta$, compactly supported near $\gamma$, such that kernel of $\hat \omega=\tilde \omega+d\eta$ has $\gamma$ as a non-degenerate periodic orbit. The pair $(\tilde \lambda, \hat \omega)$ is still a stable Hamiltonian structure because $\tilde \lambda$ is closed near $\gamma$. As we observed in the first perturbation, the one-form $\tilde \lambda$ is a stabilizing one-form for any Hamiltonian structure that is a perturbation of $\tilde \omega$ compactly supported near the periodic orbit. The Reeb field of $(\tilde \lambda,\hat \omega)$ is $C^\infty$-close to the Reeb field of $(\tilde \lambda,\tilde \omega)$, but since $\Sigma$ was a $\partial$-strong partial section of the latter, it will still be a $\delta$-strong partial section of the former.

 Doing this at each binding orbit of each partial section, we end up with a stable Hamiltonian structure, which we denote $(\hat \lambda,\hat \omega)$ abusing notation, whose Reeb field admits a positive and a negative $\partial$-strong partial sections with non-degenerate binding orbits. The two-form $\hat \omega$ is cohomologous to $\omega$ and $C^\infty$-close to it, so this concludes the proof. The isotopy between $M$ and a stable hypersurface $M_1$ with induced Hamiltonian structure $\hat \omega$ is obtained via Lemma \ref{lem:app}, and gives an isotopy that is compactly supported near the bindings orbits of the partial sections. Our observations about the stabilizing properties of $\tilde \lambda$ show that this isotopy is done through stable hypersurfaces.
\end{proof}

\subsection{Non-density of Reeb-like flows}

We have shown in the previous sections how to construct stable hypersurfaces whose characteristic foliation admits a positive and a negative $\partial$-strong partial section with non-degenerate binding orbits. This is a robust obstruction to being Reeb-like that we obtain as a combination of Proposition \ref{prop:robpartial} and Lemma \ref{lem:obs}.

\begin{corollary}\label{cor:robobs}
    Let $X$ be a vector field in $M$ that admits a positive and a negative $\partial$-strong partial section whose binding orbits are non-degenerate. Then there exists a $C^1$-neighborhood of $X$ in $\mathfrak{X}(M)$ such that every vector field in this neighborhood is not Reeb-like.
\end{corollary}

Being Reeb-like is a property of the foliation, so we have actually shown that any positive multiple of such a vector field admits a $C^1$-neighborhood in the space of vector fields in $M$ such that every vector field in that neighborhood is not Reeb-like.\\

Recall that $\mathfrak{X}_{vol}(M)$ denotes the space of non-vanishing vector fields preserving some volume form. We denote by $\mathcal{R}(M)$ and $\mathcal{SR}(M)$ the space of Reeb-like and stable-like flows in $M$ respectively. We are ready to prove Theorem \ref{thm:Reeb}, whose statement we now recall.

\begin{theorem}\label{thm:dens}
Let $M$ be a closed three-manifold. On each homotopy class of vector fields, there exists $Y\in \mathcal{SR}(M)$ and a $C^1$-neighborhood $\mathcal{V}$ of $Y$ in $\mathfrak{X}(M)$ such that $\mathcal{V}\cap \mathcal{R}(M)=\emptyset$.
\end{theorem}
A consequence is that $\mathcal{R}(M)$ is not $C^1$-dense in $\mathcal{SR}(M)$ or in $\mathfrak{X}_{vol}(M)$. One can as well fix a volume-form $\mu$ and denote by $\mathfrak{X}_{\mu}(M)$, $\mathcal{R}_\mu(M)$ and $\mathcal{SR}_\mu(M)$ the space of flows preserving $\mu$, of Reeb-like flows preserving $\mu$ and of stable-like flows preserving $\mu$ in $M$. Before proving Theorem \ref{thm:dens}, we state a straightforward corollary in terms of flows preserving a fixed volume form.
 
\begin{corollary}\label{cor:fixmu}
For any volume form $\mu$ in $M$, the set $\mathcal{R}_\mu(M)$ is not $C^1$-dense in $\mathcal{SR}_\mu(M)$ or in $\mathfrak{X}_{\mu}(M)$.
\end{corollary}
\begin{proof}
By Theorem \ref{thm:dens}, there is some $X\in \mathcal{SR}(M)$ and a neighborhood $\mathcal{V}$ of $X$ in $\mathfrak{X}(M)$ such that $\mathcal{V}\cap \mathcal{R}(M)=\emptyset$. There exists a stable Hamiltonian structure $(\lambda,\omega)$ such that $\lambda(X)>0$ and $\iota_X\omega=0$, and $X$ preserves the volume form $\hat \mu=\frac{1}{\lambda(X)}\lambda\wedge \omega$. For any other volume form $\mu=h\hat \mu$ where $h$ is a positive function in $M$, the vector field $\tilde X=\frac{1}{h}X$ is in $\mathcal{SH}_\mu(M)$, and we claim that it cannot be $C^1$-approximated by flows in $\mathcal{R}(M)$. If there is a sequence of vector fields $Y_n \in \mathcal{R}(M)$ whose limit is $\tilde X$, then $hY_n$ is a sequence of vector fields in $\mathcal{R}(M)$ whose limit is $X$, reaching a contradiction. Thus $\tilde X$ admits a neighborhood in $\mathfrak{X}(M)$ that does not intersect $\mathcal{R}(M)$, and since $\mathcal{R}_\mu(M)$ is a subset of $\mathcal{R}(M)$, we conclude that there is a neighborhood $\mathcal{U}$ of $\tilde X$ in $\mathfrak{X}(M)$ such that $\mathcal{U}\cap \mathcal{R}_\mu(M)=\emptyset$, proving the corollary.
\end{proof}
We proceed to the proof of Theorem \ref{thm:dens}.
\begin{proof}[Proof of Theorem \ref{thm:dens}]
There are several ways to construct a Reeb field of a stable Hamiltonian structure on a given homotopy class of non-vanishing vector fields on any closed three-manifold and such that $(\lambda,\omega)$ is such that $d\lambda=f\omega$ for a non-constant function $f$. Perhaps the simplest combination of results is the following. Homotopy classes of vector fields and of plane fields are in one-to-one correspondence, simply by considering the orthogonal complement with respect to some ambient Riemannian metric. Given a homotopy class of vector fields, there exists a contact structure $\xi$ in the corresponding homotopy class of plane fields, as implied by the classical results of Lutz and Martinet. Let $(\alpha,d\alpha)$ be the stable Hamiltonian structure such that $\ker \alpha=\xi$. By \cite[Theorem 3.37]{CV} the stable Hamiltonian structure $(\alpha,d\alpha)$ is homotopic to some other stable Hamiltonian structure $(\lambda,\omega)$ such that $d\lambda=f\omega$ with $f$ non-constant. Consider the symplectization of $(\lambda,\omega)$, namely $M\times (-\varepsilon,\varepsilon)$ equipped with the symplectic form
$$ \Omega= \omega+ d(t\lambda), $$
where $t$ parametrizes $(\varepsilon,\varepsilon)$. By Proposition \ref{prop:SHSht} and Theorem \ref{thm:SHSdelta}, there exists a $C^0$-small isotopy of embedded hypersurfaces $M_t=\varphi_t(M)$, such that $M_0=M\times \{0\}$ and $M_1$ is stable and the characteristic foliation of $M_1$ admit a $\partial$-strong positive and negative partial section with non-degenerate binding orbits. Choose a smooth section $X_t$ of the kernel of the family of closed two-forms $\omega_t=\varphi_t^*\Omega$. The flow $X_1$ is stable-like, since $\omega_1$ admits stabilizing one-form, and admits a $C^1$-neighborhood in $\mathfrak{X}(M)$ without Reeb-like flows by Corollary \ref{cor:robobs}. Furthermore, the vector field $X_1$ is homotopic to $X_0$ which belongs to an arbitrary homotopy class of vector fields. 
\end{proof}

\begin{Remark}
The ideas of the proof of Theorem \ref{thm:Reeb} via Theorem \ref{thm:SHSdelta} can be directly implemented by considering homotopies of stable Hamiltonian structures in $M$ instead of manipulating stable embeddings in its symplectization. Namely, we can introduce rational half-twisting regions and robust partial sections to the Reeb field of a stable Hamiltonian structure by deforming a well-chosen pair $(\lambda,\omega)$ in $M$ for instance using several techniques in \cite{CV}. However, we have to work with embeddings anyway to prove Theorem \ref{thm:main}.
\end{Remark}

\section{Generic non-integrability of volume-preserving flows}\label{s:nonint}

Throughout this section, and to apply KAM arguments, we will use the language of vector fields preserving a given volume form $\mu$. In this context, there is a one-to-one correspondence between smooth non-vanishing vector fields preserving $\mu$ and closed non-degenerate two-forms in $M$ (i.e. Hamiltonian structures). We denote by $\mathfrak{X}_\mu^\eta(M)$ the set of smooth non-vanishing vector fields in $M$ preserving $\mu$ and such that $[\iota_X\mu]=\eta\in H^2(M,\mathbb{R})$, which is equivalent to the set of Hamiltonian structures with cohomology class $\eta$, and equip it with the $C^\infty$-topology.

\subsection{The Wronskian and a KAM theorem}

Let $X$ be a smooth volume-preserving vector field on a three-dimensional manifold $M$. Let $U \subset M$ be an integrable region of $X$. By Theorem \ref{thm:Ste} there exist coordinates $(x,y,r)$ of $U\cong T^2\times I$ such that $X$ takes the form
\begin{equation}\label{eq:ste}
X= G(x,y,r)\cdot  \left( f_1(r) \pp{}{x} + f_2(r) \pp{}{y} \right). 
\end{equation}
Following \cite{KKPS}, define the Wronskian of $X$ along $U$ as
$$ W(X|_U):=\norm{f_1'(r)f_2(r)-f_1(r)f_2'(r)}.$$

\begin{defi}\label{def:twKAM}
A vector field $X$ has non-zero Wronskian in an integrable region $U$ if
$$ W(X|_U)(r)>\tau,\enspace \text{for some } \tau>0 \text{ and each } r\in [0,1]. $$
In other words, the Wronskian of $X$ in $U$ is uniformly bounded from below by a positive number.
\end{defi}

Let us now recall a KAM theorem for exact divergence-free vector fields following Khesin, Kuksin, and Peralta-Salas \cite[Theorem 3.2]{KKPS}. 

\begin{theorem}\label{thm:KAM}
Assume that $X$ is a vector field with an integrable region $U=T^2\times (-\varepsilon,\varepsilon)$ with coordinates $(x,y,r)$ where $X$ assumes the form 
$$X=\left( f(r) \pp{}{x} + g(r) \pp{}{y} \right)$$
and therefore is an exact divergence-free vector field with respect to $dx\wedge dy\wedge dr$. Assume, furthermore, that $X$ has non-zero Wronskian. Then there exists constants $C, \delta_0>0$ and a continuous function $\omega:[0,\delta_0]\rightarrow \mathbb{R}$ satisfying $\lim_{z\rightarrow 0}\omega(z)=0$ such that the following properties are satisfied. Let $Y$ be any exact divergence-free vector field with respect to $dx\wedge dy \wedge dr$ such that $\norm{Y-X}_{C^\infty}<\delta<\delta_0$. There exists a diffeomorphism $\psi:U \rightarrow U$ preserving $dx\wedge dy\wedge dr$, a Borel set $Q\subset (-\varepsilon,\varepsilon)$ and smooth functions $\tilde f(r), \tilde g(r)$ such that:
\begin{itemize}
\item[-]  $\operatorname{Leb}\left((-\varepsilon,\varepsilon)\setminus \overline Q\right)\leq C\tau^{-1}\sqrt\delta$, where $\tau$ is the lower bound of the Wronskian of $X$ in $U$,
\item[-] $\norm{\psi-\operatorname{id}}_{C^\infty}\leq \omega(\delta)$,
\item[-]  $\norm{f-\tilde f}_{C^\infty} \leq \omega(\delta)$ and $\norm{g-\tilde g}_{C^\infty}\leq \omega(\delta)$,
\item[-] for $\tilde r\in \overline Q$, the vector field $Y$ transformed by $\psi$ takes the form 
$$ \psi_*Y=\tilde f(r)\pp{}{x}+\tilde g(r)\pp{}{y}, $$
and the ratio $\frac{\tilde f(r)}{\tilde g(r)}$ is an irrational number.
\end{itemize}
\end{theorem}
The statement above does not correspond exactly to \cite[Theorem 3.2]{KKPS}. Indeed, their statement assumes less regularity of the involved vector fields, and implies the existence of a diffeomorphism $\psi$ and function $\tilde f, \tilde g$ that are only $C^1$ in general. In our context, everything is smooth and the application of Moser's twist theorem \cite{Mos} (see \cite[Remark 3]{KKPS} and \cite{Y} for more details) yields a $C^\infty$-diffeomorphism $\psi$ and $C^\infty$-functions $\tilde f, \tilde g$.\\

\subsection{Non-integrability is residual}

Consider the subset $\mathfrak{X}_\mu^\eta(M) \subset \mathfrak{X}^\mu(M)$ of divergence-free vector fields such that $[\iota_X\mu]=\eta$, where $\eta\in H^2(M,\mathbb{R})$, and we denote by $\mathcal{N} \subset \mathfrak{X}_\mu^\eta(M)$ those vector fields that do not admit any smooth first integral.  That is, if $X\in \mathcal{N}$ then $[\iota_X\mu]=\eta$ and every function $f\in C^\infty(M)$ such that $df(X)=0$ is constant. Let $\mathcal{I}$ be the subset of vector fields in $\mathfrak{X}_\mu^\eta(M)$ that admit an integrable region $U\cong T^2 \times I$ (as in Definition \ref{def:intreg}) somewhere in $M$. The following lemma is completely standard if we keep in mind that we work with non-vanishing vector fields.

\begin{lemma}\label{lem:intequiv}
The sets $\mathfrak{X}_\mu^\eta(M)\setminus \mathcal{N}$ and $\mathcal{I}$ are equal.
\end{lemma}
\begin{proof}
Assume that a vector field $X$ admits an integrable region $U\cong T^2\times I$ with coordinates $(x,y,r)$. Define a function $f:T^2\times I \longrightarrow \mathbb{R}$ as $f(x,y,r)=\phi(r)$ where $\phi:I\longrightarrow \mathbb{R}$ is a bump function equal to zero near $r=0$ and $r=1$. Then $f$ extends to a globally defined function in $M$ and is a non-constant first integral of $X$. 

Conversely, assume that $X$ admits a non-constant smooth first integral. Then if $c$ is a regular value of $f$, choose $\varepsilon>0$ small enough such that $[c-\varepsilon,c+\varepsilon]$ are all regular values of $f$. Any connected component of $f^{-1}(x)$ with $x\in [c-\varepsilon,c+\varepsilon]$ is necessarily a torus $T$: it is a closed surface, the vector field $X$ is non-vanishing and tangent to it, and the surface is cooriented by the gradient of $f$. Using an auxiliary metric, the flow $\phi_t$ of the vector field $\frac{\nabla f}{df(\nabla f)}$ satisfies near this torus $\pp{}{t}df(\phi_t(x))|_{t=s}=1$. For $t$ near $0$, this shows that $f(\phi_t(x))=f(x)+t$ for any $x\in T$, proving that the fibration induced by $f$ is trivial in the connected component of $f^{-1}([-\varepsilon,\varepsilon])$ that contains $T$. Hence, each connected component of $f^{-1}([c-\varepsilon, c+ \varepsilon])$ is diffeomorphic to $T^2\times I$. The vector field is tangent to each torus fiber, which corresponds to a connected component of each level set $f^{-1}(x)$ with $x\in [c-\varepsilon,c+\varepsilon]$, so $X$ is in $\mathcal{I}$.
\end{proof}

Hence, to show that $\mathcal{N}$ is residual in $\mathfrak{X}_\mu^\eta(M)$, it is enough to show that $\mathcal{I}$ is meager in $\mathfrak{X}_\mu^\eta(M)$.

\begin{theorem}\label{thm:residualint}
Let $M$ be a closed three-manifold with a fixed volume form $\mu$. The set $\mathcal{I}$ is meager in $\mathfrak{X}_\mu^\eta(M)$ equipped with the $C^\infty$-topology. 
\end{theorem}

\begin{proof}
We will show that there exists a set $\mathcal{A}$ that is residual in $\mathcal{I}$ and meager in $\mathfrak{X}_\mu^\eta(M)$.

\begin{prop}
The set $\mathcal{A}$ of vector fields in $\mathcal{I}$ with some integrable region $T^2\times I$ where the slope function is non-constant is residual in $\mathcal{I}$.
\end{prop}
\begin{proof}
Let $U_i$ be a countable basis of open sets of $M$ (which are not necessarily of the form $T^2\times I$). For each $U_i$, let $\mathcal{I}_i$ denote the set of vector fields in $M$ such that $U_i$ is contained in an integrable region of the flow. Similarly let $\mathcal{I}_{\operatorname{tw},i}$ denote those vector fields in $M$ such that $U_i$ is contained in an integrable region $V=T^2\times I$ for which there is some possibly smaller interval $J\subset I$ such that $(T^2\times J)\cap U_i\neq \emptyset$ and the Wronskian of $X$ in $T^2\times J$ is non-zero. Let us show that $\mathcal{I}_{\operatorname{tw},i}$ is $C^\infty$ dense in $\mathcal{I}_i$. Let $X$ be a vector field that admits an integrable region that contains $U_i$, then $U_i\subset V$ where $V\cong T^2\times (-\varepsilon,\varepsilon)$, and without loss of generality, we assume that $(T^2\times \{0\}) \cap U_i \neq \emptyset$. By Theorem \ref{thm:Ste} there are coordinates $(x,y,r)$ for which the flow has the form
\begin{equation}\label{eq:Pster}
X=G(x,y,r)\cdot  \left( f_1(r) \pp{}{x} + f_2(r) \pp{}{y} \right). 
\end{equation}
We claim that, perhaps in a smaller integrable region $\tilde V\cong T^2\times (-\delta,\delta)$ contained in $V$ and intersecting $U_i$, there are coordinates $(x',y',r')$ such that 
$$\omega= \iota_X\mu= h_2(r') dr'\wedge dx' - h_1(r') dr'\wedge dy',$$
for some function $h_1, h_2$ from $(-\delta,\delta)$ to $\mathbb{R}$.
To argue this, we first need to find an auxiliary closed one-form that evaluates positively on $X$. Choose $\tilde V=T^2\times (-\delta,\delta)$ a possibly smaller integrable region intersecting $U_i$, contained in $V$, and satisfying that the slope of the flow $X$ is contained in a very small interval in $S^1$. Then the function $\frac{(f_1(r),f_2(r))}{|(f_1,f_2)|}$ has its image in a small interval in $S^1$ too, and thus there exists a pair of real numbers $(a,b)$ such that $\langle (a,b), (f_1(r),f_2(r)) \rangle>0$ for all $r\in (-\delta,\delta)$. This implies that the closed one-form $\lambda=adx+bdy$ satisfies $\lambda(X)>0$ in $\tilde V$. The pair $(\lambda,\omega)$ defines a stable Hamiltonian structure in $\tilde V$, where the Reeb flow is integrable and parallel to $X$. By \cite[Theorem 3.3]{CV}, there are coordinates $(x',y',r')$ of $\tilde V\cong T^2\times (-\delta,\delta)$ such that the exact form $\omega$ is $T^2$-invariant, and hence can be written in the form $\iota_X\mu= h_2(r') dr'\wedge dx' - h_1(r') dr'\wedge dy'$. 

Rename the coordinates $(x', y', r')$ back to $(x,y,r)$. The volume form can be written as $\mu=F(x,y,r)dx\wedge dy \wedge dr$ for some positive function $F(x,y,r)$, and the vector field $X$ is then $X=\frac{1}{F}(h_1(r)\pp{}{x} + h_2(r)\pp{}{y})$. We can now make an arbitrarily $C^\infty$-small perturbation of $h_1,h_2$, compactly supported in $(-\delta,\delta)$, into a pair of functions $\tilde h_1,\tilde h_2$ such that $\tilde h_1' \tilde h_2 - \tilde h_1 \tilde h_2'\neq 0$ near $0 \in (-\delta,\delta)$. This can be done as follows. If the Wronskian of $X$ does not vanish near $r=0$, the vector field is already in $\mathcal{I}_{tw,i}$. So assume that the Wronskian vanishes there, and assume that $h_2(0)\neq 0$ (the case $h_1(0)\neq 0$ is analogous). We have $h_1'(0)h_2(0)-h_2'(0)h_1(0)=0$ by hypothesis. Let $\tilde h_1$ be a function $C^\infty$-close to $h_1$, which coincides with $h_1$ away from a small neighborhood of the origin, and such that $\tilde h_1(0)=h_1(0)$ and $\tilde h_1'(0)\neq 0$. It follows that $\tilde h_1'(0) h_2(0)- h_2'(0)\tilde h_1(0)\neq 0$. Choosing $\tilde h_2=h_2$, the pair $(\tilde h_1,\tilde h_2)$ is the required perturbation. 

The divergence free vector field $Y$ defined by $\iota_Y\mu = \tilde h_2 dr\wedge dx -\tilde h_2 dr\wedge dy$ is arbitrarily $C^\infty$-close to $X$. It has non-vanishing Wronskian in $T^2\times (-\delta',\delta')$ for some $\delta'<\delta$, and $\iota_Y\mu$ is exact in $V$, thus $[\iota_Y\mu]=[\iota_X\mu]$ and $Y \in \mathcal{I}_{\operatorname{tw},i}$. We have thus shown that $\mathcal{I}_{\operatorname{tw},i}$ is dense in $\mathcal{I}_i$.  \\

 We will now prove that there is an open neighborhood $\mathcal{A}_i$ of $\mathcal{I}_{\operatorname{tw},i}$ in $\mathcal{I}_i$ such that every flow in $\mathcal{A}_i$ admits an integrable region with non-constant slope. Consider a flow $X$ in $\mathcal{I}_{\operatorname{tw},i}$, hence as before there is an integrable region $V$ containing $U_i$ where $X$ takes the form \eqref{eq:Pster} and satisfies somewhere in $V\cap U_i$ that the Wronskian is non-zero. Choose some $\tilde V\cong T^2\times (-\varepsilon,\varepsilon)\subset V$ such that $X|_{\tilde V}$ has non-zero Wronskian everywhere in $\tilde V$, and $\tilde V\cap U_i \neq \emptyset$. Reparametrize $X$ to $\tilde X=H\cdot X$ where $H$ is a smooth positive function equal to $\frac{1}{G(x,y,r)}$ in $\tilde V$. Then $\tilde X$ has the form
$$ \tilde X=f_1(r)\pp{}{x}+f_2(r)\pp{}{y}, $$
in $\tilde V$ and has non-zero Wronskian. The vector field $\tilde X$ preserves the volume form $\tilde \mu = G\mu$, and the volume form $\mu$ is in these coordinates of the form $\mu=F(x,y,r)dx\wedge dy \wedge dr$. Notice that the closed two-form $\iota_{\tilde X}G\mu$ has the expression
$$\iota_{\tilde X} G\mu=f_1FGdy\wedge dr -f_2FGdx\wedge dr,$$
and arguing as before, this form vanishes when restricted to $T^2\times \{0\}$, and thus it must be exact. We claim that using that $\tilde X$ preserves $G\mu$ and that it has non-zero Wronskian implies that $G\mu$ must be of the form $\rho(r)dx\wedge dy \wedge dr$. This can be seen as follows. Write $G\mu=\rho(x,y,r) dx\wedge dy \wedge dr$, where $\rho=F\cdot G$. The closedness of $\iota_{\tilde X}\rho dx\wedge dy \wedge dr$ and of $\iota_{\tilde X}dx\wedge dy \wedge dr$ imply that $\iota_XdH=0$ and hence that $H$ is a first integral of $X$. On the other hand, there is a dense set of values $a\in I$ for which $X$ is an irrational rotation on the torus $T^2\times \{a\}$, a flow which does not admit any first integral. This implies that the function $H$ must be constant on each of these tori and hence in each torus fiber, which means that $H=H(r)$. 

Changing the coordinate $r$ in $I$ we can then assume that $\tilde \mu=dx\wedge dy\wedge dr$. We are in the hypotheses of Theorem \ref{thm:KAM}, so there exists $\delta_0>0$ such that for $\delta<\delta_0$, any $\tilde \mu$-exact divergence-free vector field $Z$ satisfying
$$ \norm{Z-\tilde X}_{C^\infty} <\delta $$
satisfies the conclusions of the KAM theorem. Let $Y\in \mathcal{I}_i$ be a vector field such that
$$ \norm{Y-X}_{C^\infty}<\delta_1. $$
Observe that since $[\iota_Y\mu]=[\iota_X\mu]=\eta$, the vector field $Y$ is also exact in $T^2\times I$.
Taking $\delta_1$ small enough (this only depends on the $C^\infty$ norm of $H$), it implies that
$$ \norm{H\cdot Y-\tilde X}_{C^\infty}<\delta. $$
Furthermore $\tilde Y=H\cdot Y$ satisfies $\iota_{\tilde Y}G\mu=\iota_Y\mu$, and since $[\iota_Y\mu]=[\iota_X\mu]$, it follows that $\tilde Y$ is exact in $\tilde V$ with respect to $dx\wedge dy \wedge dt$.

 If $\delta$ is small enough, the theorem applies to $\tilde X$ and $\tilde Y$, so there exists a Borel set $Q\subset (-\varepsilon,\varepsilon)$ and smooth functions $\tilde f(t), \tilde g(t)$ satisfying the conclusions of Theorem \ref{thm:KAM}. Since the Borel set $Q$ is of almost full measure, and the rotation vector $(\tilde f(t),\tilde g(t))$ is close to the rotation vector $(f(t),g(t))$ in the invariant tori, we deduce that $\tilde Y$ inherits a set of almost full measure of tori that have different irrational rotation vectors. By assumption $Y$ admits an integrable region $V_2$ that contains $U_i$, and the invariant tori inherited from the KAM theorem are minimal, so these tori are necessarily fibers of the integrable region $V_2$ of $\tilde Y$. In the integrable region $V_2$ intersecting $U_i$, the slope of $\tilde Y$ cannot be constant, since there are invariant tori with different irrational rotation vectors. Since $Y$ is parallel to $\tilde Y$, the same conclusion holds for $Y$. Thus, we have shown that there exists an open neighborhood $\mathcal{A}_i$ of $\mathcal{I}_{\operatorname{tw},i}$ inside $\mathcal{I}_i$ such that every element in $\mathcal{A}_i$ has an integrable region that intersects $U_i$ and has non-constant slope near some torus intersecting $U_i$. Since $\mathcal{I}_{\operatorname{tw},i}$ is dense in $\mathcal{I}_i$, the set $\mathcal{A}_i$ is open and dense and thus the complement $\mathcal{I}_i\setminus \mathcal{A}_i$ is nowhere dense in $\mathcal{I}_i$. In particular
$$ \mathcal{B}=\bigcup_i \mathcal{I}_i\setminus \mathcal{A}_i  $$
is meager in $\mathcal{I}$. Notice that $\mathcal{I}\setminus \mathcal{A} \subset \mathcal{B}$, and thus $\mathcal{I}\setminus \mathcal{A}$ is also meager in $\mathcal{I}$. We conclude that $\mathcal{A}$ is residual in $\mathcal{I}$.
\end{proof}

To conclude, we will show that $\mathcal{A}$ is meager in $\mathfrak{X}_\mu^\eta(M)$.  Let $\mathcal{P}$ be the set of vector fields preserving $\mu$ such that every periodic orbit is non-degenerate. Recall that a periodic orbit is non-degenerate if the linearized Poincar\'e return map does not have any root of unity as an eigenvalue. In particular, a non-degenerate periodic orbit is isolated. Notice that any element in $\mathcal{A}$ admits an invariant torus with a rational rotation vector, which is filled with non-isolated and hence degenerate periodic orbits. Thus, if we show that the set $\mathcal{P}^\eta=\mathcal{P}\cap \mathfrak{X}_\mu^\eta(M)$ is residual in $\mathfrak{X}_\mu^\eta$, we will conclude that $\mathcal{A}$ is meager.

Robinson \cite{Rob} proved that $\mathcal{P}$ is residual in $\mathfrak{X}^\mu(M)$, so $\mathcal{P}=\bigcap_{N \in \mathbb{N}} \mathcal{P}_N$ for some open and dense sets $\mathcal{P}_N$ in $\mathfrak{X}^\mu(M)$. Hence $\mathcal{P}_N^\eta=\mathcal{P}_N\cap \mathfrak{X}_\mu^\eta(M)$ is open in $\mathfrak{X}_\mu^\eta(M)$. It remains to show that $\mathcal{P}_N^\eta$ is dense in $\mathfrak{X}_\mu^\eta(M)$ to conclude that $\mathcal{P}^\eta=\bigcap_{N=1}^\infty \mathcal{P}_N^\eta$ is residual in $\mathfrak{X}_\mu^\eta$. To see that $\mathcal{P}_N^\eta$ is dense, we will directly show that $\mathcal{P}^\eta$ is dense by a symplectization procedure. Let $X$ be a divergence-free vector field such that $[\iota_X\mu]=\eta$ is an exact non-degenerate two-form. Endow $M\times (-\varepsilon,\varepsilon)$ with the symplectic form
$$ \omega= d(t\lambda)+\iota_X\mu, $$
where $\lambda$ is some one-form such that $\lambda(X)=1$ and $t$ is the coordinate in $(-\varepsilon,\varepsilon)$. The vector field $X$ coincides with the Hamiltonian vector field defined by the function $H=t$ along $M\times \{0\}$. By \cite[Theorem 1]{Rob} there is an arbitrarily $C^\infty$-close function $\tilde H$ such that the Hamiltonian flow along $\{\tilde H=0\}\cong M$ is non-degenerate. Since $\{\tilde H=0\}$ is $C^\infty$-close to $M\times \{0\}$, the restriction of $\omega$ to $\{\tilde H=0\}$ is a non-degenerate closed two-form $\tilde \omega$ that is $C^\infty$-close to $\iota_X\mu$ and $[\tilde \omega]=\eta$. We conclude that the vector field $Y$ in $M$ defined by the equation $\iota_Y\mu=\eta$ is arbitrarily $C^\infty$-close to $X$ and lies in $\mathcal{P}^\eta$. We have thus shown that $\mathcal{P}_N^\eta$ is open and dense in $\mathfrak{X}_{\mu}^\eta(M)$ for each $N$, and we conclude that $\mathcal{P}^\eta=\bigcap_{N=1}^\infty \mathcal{P}_N^\eta$ is residual in $\mathfrak{X}_{\mu}^\eta(M)$.

Since any element in $\mathcal{A}$ admits an invariant torus with a rational rotation vector, and non-degenerate periodic orbits are necessarily isolated, we deduce that $\mathcal{A}\subset \mathfrak{X}_{\mu}^\eta(M) \setminus \mathcal{P}$ and hence $\mathcal{A}$ is meager in $\mathfrak{X}_{\mu}^\eta(M)$. \\

We are ready to show that $\mathcal{I}$ is meager in $\mathfrak{X}_\mu^\eta(M)$, which follows from the properties of residual and meager sets. We include the argument for completeness. We have shown that $\mathcal{A}$ is meager in $\mathfrak{X}_\mu^\eta(M)$, hence 
$$ \mathcal{A}=\bigcup_i A_i, \enspace A_i \subset \mathfrak{X}_\mu^\eta \text{ nowhere dense in } \mathfrak{X}_\mu^\eta(M). $$
On the other hand, $\mathcal{I}\setminus \mathcal{A}$ is meager in $\mathcal{I}$ so
$$\mathcal{I}\setminus \mathcal{A}= \bigcup_i A_i', \quad A_i' \subset \mathcal{I} \text{ nowhere dense in } \mathcal{I}.$$
Since $A_i'$ is nowhere dense in $\mathcal{I}$, it is nowhere dense in $\mathfrak{X}_\mu^\eta(M)$ too and we have
$$\mathcal{I}=\left(\bigcup_i A_i \right) \cup \left(\bigcup_i A_i'\right), \enspace A_i', A_i \subset \mathfrak{X}_\mu^\eta(M) \text{ nowhere dense in } \mathfrak{X}_\mu^\eta(M).$$
We conclude that $\mathcal{I}$ is meager in $\mathfrak{X}_\mu^\eta(M)$, finishing the proof of the theorem.
\end{proof}
Combining Theorem \ref{thm:residualint} and Lemma \ref{lem:intequiv}, we deduce that a generic element in $\mathfrak{X}_\mu^\eta$ admits no smooth first integral. We have restricted to smooth vector fields but probably a similar statement holds for $C^k$-vector fields endowed with the $C^k$-topology, as long as $k$ is large enough to apply the KAM theorem, for example $k\geq 4$ suffices. \\

Theorem \ref{thm:residualint} holds as well if we consider the whole set of (possibly somewhere vanishing) volume-preserving vector fields, which we denote by $\mathcal{X}_{\eta}^\mu(M)$ to distinguish it from the set $\mathfrak{X}_\mu^\eta(M)$. The only difference in the proof is that one has to use that vector fields in $\mathfrak{X}_\mu^\eta(M)$ with isolated zeroes define an open and dense set, and that Lemma \ref{lem:intequiv} holds in this open set. A particularly interesting case from a hydrodynamics point of view is the set of exact divergence-free vector fields, i.e. when $[\eta]=[0]$.

\begin{corollary}\label{cor:exnonint}
Let $M$ be a closed three-manifold, and $\mathfrak{X}_{0}^\mu(M)$ the set of exact divergence-free vector fields with respect to a volume form $\mu$. The set of flows that admit no smooth first integral is residual in $\mathfrak{X}_{0}^\mu(M)$.
\end{corollary}

As discussed in \cite[Remark p 2039]{EPST}, Corollary \ref{cor:exnonint} in the whole set of exact divergence-free vector fields can be used to extend the proof of the main theorem in \cite{EPST} about the uniqueness properties of the helicity in the set of exact divergence-free vector fields of $C^1$-regularity to those vector fields of $C^\infty$-regularity, and similarly for $C^k$-regularity with $k\geq 4$. The uniqueness properties of the helicity for smooth or $C^k$ volume-preserving vector fields for any $k\geq 4$ were established in \cite{KPSY} using a different approach. \\

Theorem \ref{thm:residualint} has a symplectic interpretation in terms of closed non-degenerate two-forms and embedded hypersurfaces.
\begin{corollary}\label{cor:HamNoInt}
On a closed three-dimensional manifold $M$, the characteristic foliation of a generic Hamiltonian structure within the set of Hamiltonian structures with a given cohomology class in $H^2(M,\mathbb{R})$ admits no smooth first integral. On a four-dimensional symplectic manifold $(W,\Omega)$, the characteristic foliation of a generic embedded hypersurface admits no smooth first integral.
\end{corollary}
\begin{proof}
    The statement for Hamiltonian structures is immediate from the fact that elements in $\mathfrak{X}_\mu^\eta (M)$ and Hamiltonian structures with cohomology class $\eta$ are identified by contraction with $\mu$. The statement for embedded hypersurfaces follows from three observations. The first is that two isotopic embedded hypersurfaces $i_0:M_0\rightarrow W$ and $i_1:M_1\rightarrow W$ in $(W,\Omega)$ have cohomologous Hamiltonian structure $\omega_0=i_0^*\Omega$ and $\omega_1=i_1^*\Omega$. The second is that $C^\infty$-close hypersurfaces have $C^\infty$-close induced Hamiltonian structures, and the third (and perhaps less immediate one) is that a sufficiently $C^\infty$-small exact perturbation $\tilde \omega$ of a Hamiltonian structure $\omega$ induced in a hypersurface $M$ is induced in a hypersurface $\tilde M$ that is $C^\infty$-close to $M$, as per Lemma \ref{lem:app} in the Appendix. These observations imply that an open and dense property in $\mathfrak{X}_\mu^\eta(M)$ is also an open and dense property in the set of Hamiltonian structures with cohomology class $\eta$ induced on embedded hypersurfaces in $(W,\Omega)$, and thus also open and dense in the set of embeddings. Thus, given the set $A$ of embedded hypersurfaces within an isotopy class (whose induced Hamiltonian structures are in some cohomology class $\eta$, for example), any residual property in $\mathfrak{X}_\mu^\eta(M)$ is also residual in $A$. This holds for any isotopy class and thus proves the statement.
\end{proof}

\section{Stability is not open}\label{s:notgeneric}

We proceed to the proof of the main Theorem \ref{thm:main}. We will first prove it under a simplifying assumption, and then show that one can reduce to this case by a preliminary $C^1$-perturbation of the hypersurface.

\subsection{Proof under a simplifying assumption}\label{ss:mainsimp}
The goal of this section is to prove Theorem \ref{thm:main} under the simplifying assumption that the stable Hamiltonian structure $(\lambda,\omega)$ has a non-constant proportionality factor $d\lambda/\omega$. In the final section, we will reduce to this case. Recall that $\mathcal{HS}(W)$ denotes the space of embedded hypersurfaces in $W$.

\begin{theorem}\label{thm:mainsimp}
Let $M\rightarrow (W,\Omega)$ be a closed stable hypersurface in a four-dimensional symplectic manifold. Assume that an induced stable Hamiltonian structure $(\lambda,\omega)$ is such that $d\lambda=f\omega$ with $f$ non-constant. Then there exists a stable hypersurface $\tilde M\subset W$, stable isotopic and $C^0$-close to $M$, and a $C^\infty$-open neighborhood $\mathcal{U}$ of $\tilde M$ in $\mathcal{HS}(W)$ such that a generic element in $\mathcal{U}$ is not stable.
\end{theorem}

\begin{proof}[Proof of Theorem \ref{thm:mainsimp}]
Let $M$ be a stable embedded hypersurface in $(W,\Omega)$ and $(\lambda,\omega)$ an induced stable Hamiltonian structure in $M$ which by assumption can be chosen such that $f=\frac{d\lambda}{\omega}$ is not constant. By Proposition \ref{prop:SHSht} and Theorem \ref{thm:SHSdelta}, there exists a stable isotopic $C^0$-close hypersurface $\tilde M\subset W$ such that the characteristic foliation of $\tilde M$ admits a positive and a negative $\partial$-strong partial section with non-degenerate binding periodic orbits. Let $(\tilde \lambda,\tilde \omega)$ be a stable Hamiltonian structure on $\tilde M$, where $\tilde \omega$ is the pullback of $\Omega$ by the inclusion of $\tilde M$ in $W$. Then, by Proposition \ref{prop:robpartial} there exists a neighborhood $\mathcal{U}\subset \mathcal{HS}(W)$ of $\tilde M$ such that the characteristic foliation of every hypersurface in $\mathcal{U}$ admits a positive and a negative partial section. The characteristic foliation of any hypersurface in $\mathcal{U}$ is not Reeb-like by Corollary \ref{cor:robobs} and does not admit a global cross-section by Proposition \ref{prop:obsGCS}.

This shows that if an element of $\mathcal{U}$ is stable, then any defining stable Hamiltonian structure $(\tilde \lambda,\tilde \omega)$ is such that $d\tilde \lambda=f \tilde \omega$ with $f$ non-constant. Indeed, if $f$ is constant and $f\neq 0$ then $X$ is Reeb-like, and if $f\equiv 0$ then $X$ admits a cross-section. By Corollary \ref{cor:HamNoInt} there exists a residual set $\mathcal{A}$ in $\mathcal{U}$ such that the characteristic foliation of an element in $\mathcal{A}$ admits no smooth first integral. Hence, a section $X$ of the characteristic of a hypersurface in $\mathcal{A}$ is not Reeb-like, admits no cross-section, and admits no smooth first integral. We conclude that such hypersurface is necessarily unstable, concluding the proof.
\end{proof}

The previous proof applies as well to the flows we constructed in Theorem \ref{thm:Reeb} and shows that stable-like flows are neither $C^\infty$-open nor locally generic in $\mathfrak{X}_\mu^\eta(M)$. Since $\mathfrak{X}_\mu^\eta(M)$ is endowed with the subspace topology, we deduce as well that stable-like flows are not $C^\infty$-open in $\mathfrak{X}_\mu(M)$. It is not clear, however, if stable-like flows are generic when we don't fix the cohomology class of $\iota_X\mu$, i.e. if they are generic in $\mathfrak{X}_\mu(M)$ for an arbitrary three-manifold $M$. We know it is not the case on rational homology spheres, since there $\mathfrak{X}_\mu(M)=\mathfrak{X}_\mu^0(M)$.

We can also deduce that a larger class of volume-preserving flows on a closed three-manifold is not $C^\infty$-open. 
\begin{defi}[\cite{PRT}]
A volume-preserving vector field $X$ is Eulerisable if there exists an ambient metric in $M$ such that $X$ is a stationary solution to the Euler equations for the metric $g$ and some volume form $\mu$.
\end{defi}
The property of being Eulerisable is equivalent to the existence of a one-form $\alpha$ such that $\alpha(X)>0$ and $\iota_Xd\alpha$ is exact.

\begin{corollary}
On any closed three-manifold $M$, the set of non-vanishing Eulerisable flows preserving $\mu$ is not $C^\infty$-open in $\mathfrak{X}_\mu(M)$ or in $\mathfrak{X}_\mu^\eta(M)$.
\end{corollary}
\begin{proof}
    On any three-manifold, it is easy to construct a stable Hamiltonian structure $(\lambda',\omega')$ with a non-constant proportionality factor $\frac{d\lambda'}{\omega'}$, but to be concrete we can directly use \cite[Theorem 3.37]{CV}. The Hamiltonian structure $\omega'$ is induced by the inclusion of $M$ into the symplectization of $(\lambda',\omega')$, i.e. the symplectic manifold $M\times (-\varepsilon,\varepsilon)$ equipped with the symplectic form $d(t\lambda')+\omega'$ where $t$ denotes the coordinate on $\mathbb{R}$. By Theorem \ref{thm:mainsimp}, we can find a hypersurface $\tilde M\cong M$ isotopic $M$ whose induced Hamiltonian structure $\tilde \omega$ admits a stabilizing one-form $\lambda$ with the following property. There is a neighborhood of $\tilde \omega$ in the set of cohomologous Hamiltonian structures such that every element in that neighborhood has a characteristic foliation that is not Reeb-like and does not admit a global cross-section. Let $Y$ be the vector field defined by $\iota_Y\mu=\tilde \omega$, and notice that $Y$ is Eulerisable. Indeed, as $\tilde \omega$ is stabilizable it admits a stabilizing one-form, which satisfies $\lambda(Y)>0$ and $\iota_Yd\lambda=0$ (which is, in particular, an exact one-form). On the other hand, a generic perturbation $\tilde Y$ of $Y$ in $\mathfrak{X}_\mu^\eta(M)$ is a vector such that $\iota_Y\mu$ is not stabilizable and which admits no smooth first integral. Let us show that $\tilde Y$ is not Eulerisable. If it is not the case, then there exists a one-form $\tilde \lambda$ such that $\tilde \lambda(Y)>0$ and $\iota_Y \tilde \lambda=-dB$ for some function $B\in C^\infty(M)$. This function is a first integral of $Y$, since integrating the previous equality we have $\iota_YdB=0$. It must then be constant and $\tilde \lambda$ is a stabilizing one-form of $\iota_Y\mu$, which is a contradiction.
\end{proof}

\subsection{Adding integrability to stable hypersurfaces}\label{ss:addint}

In the previous section, we proved the main theorem for a stable hypersurface that admits a defining stable Hamiltonian structure $(\lambda,\omega)$ such that $d\lambda$ is not a constant multiple of $\omega$. In this section, we show that if this is not the case, we can find a stable hypersurface isotopic and $C^1$-close to a given one that satisfies it.

\begin{theorem}\label{thm:fnonconst}
Let $M\subset (W,\Omega)$ be a stable hypersurface. Then there exists a stable isotopic $C^1$-close hypersurface $\tilde M$ with a defining stable Hamiltonian structure $(\lambda,\omega)$ such that $d\lambda=f\omega$ with $f$ non-constant.
\end{theorem} 

It is clear that Theorem \ref{thm:fnonconst} and Theorem \ref{thm:mainsimp} imply our main result Theorem \ref{thm:main}.
We will prove the statement above by a concatenation of reductions. The following lemma establishes that if $M$ is a stable hypersurface whose characteristic foliation admits an integrable region, then we can find a stabilizing one-form such that $\frac{d\lambda}{\omega}$ is non-constant. Thus, we reduce the problem to finding a stable isotopic hypersurface with an integrable region, which will do next.

\begin{lemma}\label{prop:intSHS}
Let $(\lambda,\omega)$ be a stable Hamiltonian structure whose Reeb field admits an integrable region. Then there exists a stabilizing one-form $\tilde \lambda$ of $\omega$ such that $d\tilde \lambda=f\omega$ with $f$ a non-constant function.
\end{lemma}

\begin{proof}
We will assume that $d\lambda=c\omega$ for some constant $c$, since otherwise we are done. The fact that the Reeb field admits an integrable region implies by \cite[Theorem 3.3]{CV} that in this region $U\cong T^2\times I$ there are coordinates $(x,y,r)$ such that 
$$\omega=h_1'(r) dr\wedge dx + h_2'(r)dr\wedge dy.$$
 We can now use part of an argument that we used in the proof of Theorem \ref{thm:SHSdelta}. There is a non-zero function $\rho:I\rightarrow \mathbb{R}$ vanishing for $r\in [0,\delta]\cup [1-\delta,1]$ and satisfying $\int_0^1 \rho(r)h_1'(r)dr=\int_0^1 \rho(r) h_2'(r)dr=0$. Let $\tilde g_1(r)$ and $\tilde g_2(r)$ be primitives of $\rho h_1'$ and $\rho h_2'$ and consider the one-form 
$$\tilde \lambda= \lambda + \tilde g_1 dx + \tilde g_2 dy, $$
which is a stabilizing one-form of $\omega$ if $\rho$ is small enough. We have $d\tilde \lambda= d\lambda + \rho \omega$, and hence $\frac{d\tilde\lambda}{\omega}$ is not constant.
\end{proof}

To be able to apply the previous lemma, we need to establish that one can introduce integrable regions in the characteristic foliation of a stable hypersurface by means of a small stable isotopy.

\begin{theorem}\label{thm:addintS}
Let $M\subset (W,\Omega)$ be a stable hypersurface with a stabilizing one-form $\lambda$ satisfying $d\lambda=c\omega$ for some constant $c\in \mathbb{R}$. Then there exists a stable hypersurface $\tilde M$, stable isotopic and arbitrarily $C^1$-close to $M$, whose characteristic foliation admits an integrable region.
\end{theorem} 

To prove this theorem, the idea is to construct an integrable region near a periodic orbit (which we will show can be assumed to exist) of the characteristic foliation. This is the content of the next proposition.

\begin{prop}\label{prop:addint}
Let $M$ be a hypersurface in a symplectic manifold $(W,\Omega)$ of dimension four. Let $\gamma$ be a periodic orbit of the characteristic foliation of $M$. Then there exists an isotopic $C^1$-close hypersurface $\tilde M$ whose characteristic foliation admits $\gamma$ as a periodic orbit and has an integrable region. 
\end{prop}

\begin{proof}
Let $\iota:M\rightarrow (W,\Omega)$ be the hypersurface, and $\omega=\iota^*\Omega$ the induced Hamiltonian structure. Let $X$ be a section of $\ker \omega$ given by fixing a Hamiltonian function $H\in C^\infty(W)$ such that $M$ is a regular connected component of $H^{-1}(0)$. In other words $X$ is the Hamiltonian vector field of $H$ along $M$. The closed curve $\gamma$ is a periodic orbit of $X$. Construct Hamiltonian flow-box coordinates (see e.g. \cite[Theorem 6.3.2]{MO}) near some point $p \in \gamma$, i.e. a neighborhood $U\cong D\times [0,1]\times[-1,1]$ with coordinates $(x,y,z,t)$ such that the Hamiltonian is $H=t$, the symplectic form is $\Omega= dx\wedge dy + dz\wedge dt$ and $p=(0,0,0,0)$. In these coordinates, the Hamiltonian flow of $H$ is $X_H=\pp{}{z}$, and by construction $X_H|_M=X$.

In a small enough disk $D_0\subset D$ of radius $\delta$, we can find a function $\tau:D_0 \times \{1\}\times \{0\} \rightarrow \mathbb{R}$ such that $\phi_X^{\tau(q)}(q)\in D\times \{0\}\times \{0\}$ for each $q\in D_0 \times \{1\}\times \{0\}$, where $\phi_X^t$ denotes the flow of $X_H$. If we define $F(q)=\phi_X^{\tau(q)}(q)$, its projection to the $D$ factor can be interpreted as the first-return map $F:D_0\longrightarrow F(D_0)$ of $X_H$ along $D_0$. The map $F$ is symplectomorphism of $D_0$ into its image, where we equip $\mathbb{R}^2$ with the standard symplectic form $\omega=dx\wedge dy$, and has the origin as a fixed point (as the orbit of $X$ through $p$ is periodic). The map $F$ is $C^1$-close to its linearization $DF|_0:\mathbb{R}^2\rightarrow \mathbb{R}^2$ if the radius of $D_0$ is small enough, and similary with $F^{-1}$ and its linearization. Hence $g=DF|_0 \circ F^{-1}$ is arbitrarily $C^1$-close to the identity if we restrict to a sufficiently small neighborhood of the origin.\\

 We want first to perturb $M$ so that the new characteristic foliation admits a first-return map along $\gamma$ that is equal to $g\circ F$ in a small enough disk centered around $p$. A classical construction using generating functions, see e.g. \cite[Lemma 2.10]{AM}, shows that given any $\varepsilon>0$ there is a compactly supported symplectomorphism $g_c$ of $D_0$ into its image such that $\norm{g_c}_{C^1}<\varepsilon$ and $g_c=g$ in a small enough disk centered at the origin. 

An isotopy generating $g_c$ is given by the family of compactly supported symplectomorphisms $g_s$ obtained via the generating functions
$$ W_s(x',y)=x'y+\phi(s)V(x',y), $$
where $V(x',y)$ is such that $x'y+V(x',y)$ is a generating function of $g_c$ and $\phi(s)$ is a smooth cut-off function equal to $0$ near $s=0$ and equal to $1$ near $s=1$. Notice that $V$ is $C^2$-small, and thus $W(s,x',y):=W_s(x',y)$ satisfies that $\pp{W}{s}$ is $C^1$-small. The isotopy $g_s$ is generated by a time-dependent vector field $X_s$ such that $\iota_{X_s}dx\wedge dy=-dK_s$, where $K_s$ is a compactly supported function equal to $0$ near $s=0$ and $s=1$. The Hamilton-Jacobi equation (see e.g. \cite{BP}) and the $C^1$-smallness of $\pp{W}{s}$ imply that the function $K(x,y,z):=K_z(x,y)$ is $C^1$-small, i.e. its $C^1$-norm is bounded by $\varepsilon$ up to a constant.

Consider the function 
$$ \hat H= t+ K(x,y,z),$$ 
which is arbitrarily $C^1$-close to $H=t$ in $U$. The level set $\hat H=0$ extends as $H=0$ away from $U$, since $K$ has compact support in $D_0 \times [0,1]$. Denote by $\hat M$ this hypersurface, which is isotopic and $C^1$-close to $M$. The Hamiltonian vector field of $\hat H$ is 
 $$X_{\hat H}= \pp{}{z} - \pp{K_z}{z}\pp{}{t} + X_z,$$
 where $X_z$ and $K_z$ are the parametric families introduced before, parametrized by $z$. Using the trivial parametrization of $\{ \hat H=0\}$ by $(x,y,z)\mapsto (x,y,z,-K(x,y,z))$, which provides a diffeomorphism of $\tilde M$ with $M$, the vector field $X_{\hat H}$ along $\hat H=0$ is given by $\pp{}{z} + X_z$. Under this identification, there is a periodic orbit $\tilde \gamma\subset \tilde M$ of the characteristic foliation of $\tilde M$ that is identified with $\gamma\subset M$. The Poincar\'e return map of $\tilde X=X_{\hat H}|_{\tilde M}$ at $D_0$ is now given by $g_c\circ F$, which close to the periodic orbit is just given by an element in $SL(2,\mathbb{R})$.  \\
 
Notice that if the matrix is elliptic, then the characteristic foliation along $\hat M$ admits an integrable region: every concentric circle of small enough radius around $p$ is invariant by the first return map. In any case, we will show that one can always achieve that the first return map is the identity in a small neighborhood of the periodic orbit $\gamma$.
 
We will reuse the same notations as before in this new perturbation. We choose flow-box coordinates $U\cong D^2\times [0,1]\times [-1,1]$ of some point $p$ on the closed curve $\tilde \gamma \subset \tilde M$. Fix a small enough disk $D_0\times \{1\}\times \{0\}$ such that the flow of $\tilde X$ applied to this disk hits $D\times \{0\}\times \{0\}$ at some disk $D_1\times \{0\}\times\{0\}$. If $D_0$ is small enough, the disk $D_1$ is just the image of $D_0$ by a matrix $A\in SL(2,\mathbb{R})$. Since the symplectic linear group is connected, there is a path of matrices $A_s\in SL(2,\mathbb{R})$ such that $A_0=\operatorname{Id}$ and $A_1=A^{-1}$. The Hamiltonian
\begin{equation}
h_s (x,y)= -\frac{1}{2}\langle \dot A_s {A_s}^{-1} \cdot (y,-x), (x,y) \rangle
\end{equation}  
generates $A^{-1}$, i.e. the Hamiltonian vector field defined by $\iota_{X_s}dx\wedge dy=-d h_s$ generates an isotopy $\phi_t$ of the disk such that $\phi_t=A_t$ and hence $\phi_1=A^{-1}$. We multiply $h_s$ by a cut-off function $\psi(r)$ equal to $0$ near $r=\delta$ and equal to $1$ near $r=0$, where $r$ denotes a radial coordinate in $D$. For any small enough $\nu$, we can assume that $\norm{\psi}_{C^1}\leq C_\psi \frac{1}{\delta-\nu}$ for some constant $C_{\psi}$. Then the Hamiltonian $\tilde h_s=\psi h_s$ is compactly supported, and its norm for any fixed value of the parameter is
\begin{align*}
\norm{\tilde h_s}_{C^1} &\leq C_1 (\norm{\psi}_{C^1}\norm{h_s}_{C^0}+ \norm{\psi}_{C^0}\norm{h_s}_{C^1})\\
						&< C_2\left(\frac{r^2}{\delta-\nu}+r\right),
\end{align*}
for some constants $C_1,C_2$. For any $\delta$, we can choose $\nu$ so that there exists a universal constant $C_3$ with
$$\norm{\tilde h_s}_{C^1}\leq C_3 r.$$
Hence, given any $\varepsilon$ we can choose $\tilde h_s$ with $\norm{\tilde h_s}_{C^1}\leq \varepsilon$. Similarly, the time derivative $\dot{h}_s$ is also arbitrarily small: its norm is bounded by $C_4r^2$, for some constant $C_4$, so we can choose an arbitrarily small $\delta$ and make this norm arbitrarily small. Proceeding as before, we consider the Hamiltonian function
$$ \hat H_2= t+ \tilde h(x,y,t),$$
where $\tilde h(x,y,t):=\tilde h_t(x,y)$. This function is arbitrarily $C^1$-close to $H$, and the characteristic foliation through $\hat H_2=0$ admits an integrable region near the periodic orbit containing $p$: indeed, by construction, the first return map near $p$ is now equal to the identity.
\end{proof}
We can now prove Theorem \ref{thm:addintS} and finally Theorem \ref{thm:fnonconst}.
\begin{proof}[Proof of Theorem \ref{thm:addintS}]
Let $(\lambda,\omega)$ be a defining stable Hamiltonian structure of $M\subset (W, \Omega)$, and assume that $\frac{d\lambda}{\omega}=c$ is constant, since otherwise we are done. We treat separately the cases $c=0$ and $c\neq 0$. If $c=0$, then the Reeb field of $(\lambda,\omega)$ admits a global cross-section. This property is stable under $C^0$-perturbations of the characteristic foliation, hence it is stable under $C^1$-perturbations of the hypersurface. If $c\neq 0$, then $M$ is a contact type hypersurface. This property is also stable under $C^1$-perturbations, this is easily inferred from the fact that the contact type condition is given by the existence of a Liouville vector field $Y \in \mathfrak{X}(W)$ transverse to the hypersurface. Hence, if we show that there exists an arbitrarily small $C^1$-perturbation of $M$ into a hypersurface whose characteristic foliation admits an integrable region with non-constant slope, then the resulting hypersurface and isotopy $M_t$ are stable. Indeed, in the case of a global cross-section, the starting stabilizing one-form stabilizes as well any Hamiltonian structure near $\omega$, and if $M$ is of contact type, the isotopy $M_t$ can be stabilized by a family of contact primitives of the Hamiltonian structures induced in $M_t$.

If our starting stable hypersurface has no periodic orbit at all, we can apply the Hamiltonian $C^1$-closing Lemma \cite{PR} to slightly $C^2$-perturb $M$ to another (stable) hypersurface $\tilde M$ whose characteristic foliation admits at least one periodic orbit. This hypersurface will be stable isotopic to $M$, as we argued before. Applying Proposition \ref{prop:addint} to one of the periodic orbits then implies Theorem \ref{thm:addintS}.
\end{proof}

We proceed with the proof of Theorem \ref{thm:fnonconst}.
\begin{proof}[Proof of Theorem \ref{thm:fnonconst}]
Let $M \subset (W,\Omega)$ be a stable hypersurface and $\lambda$ a stabilizing one-form of the Hamiltonian structure $\omega$ induced in $M$. We assume that $\frac{d\lambda}{\omega}$ is constant since otherwise we are done. By Theorem \ref{thm:addintS} the hypersurface $M$ is stable homotopic to some stable hypersurface $M_1$ whose characteristic foliation admits an integrable region. Let $\lambda_1$ be the stabilizing one-form of $\omega_1$ (the Hamiltonian structure induced in $M_1$) obtained by this homotopy. By Lemma \ref{prop:intSHS} there is another stabilizing one-form $\tilde \lambda_1$ of $\omega_1$ such that $\frac{d\tilde \lambda_1}{\omega_1}$ is non-constant. Any two stabilizing one-forms are homotopic, and thus $M_1$ endowed with $(\tilde \lambda_1,\omega_1)$ is homotopic to $M$ endowed with $(\lambda,\omega)$, as we wanted to prove.
\end{proof}

We finish this section by proving that Theorem \ref{thm:main} implies that the set $\mathcal{SR}_\mu(M)$ of stable-like flows preserving a given volume form $\mu$ is not open in $\mathfrak{X}_\mu(M)$.

\begin{corollary}\label{cor:notopen}
    Let $M$ be a closed three-manifold. Then $\mathcal{SR}_\mu (M)$ is not open in $\mathfrak{X}_\mu(M)$.
\end{corollary}
\begin{proof}
    Let $\alpha$ be a contact form in $M$ and $W=M\times (-\varepsilon,\varepsilon)$ the symplectization of $M$, which is hence equipped with the symplectic form $\Omega=d(t\alpha)$. By Theorem \ref{thm:main}, there is an embedding of $M$ in $W$ that is $C^0$-close to the zero section with an induced Hamiltonian structure $\omega$, homotopic to $d\alpha$, that is stabilizable but admits arbitrarily $C^\infty$-close Hamiltonian structures which are not stabilizable. The vector field $Y$ defined by $\iota_Y\mu=\omega$ preserves $\mu$ (because $\omega$ is closed) and is stable-like because $\omega$ admits a stabilizing one-form, and hence $Y\in \mathcal{SR}_\mu(M)$. Choosing an arbitrarily $C^\infty$-close non-stabilizable Hamiltonian structure $\tilde \omega$, the vector field defined by $\iota_X\mu=\tilde \omega$ preserves $\mu$, is not stable-like and is arbitrarily $C^\infty$-close to $Y$. This shows that any neighborhood of $Y$ in $\mathcal{X}_\mu(M)$ admits elements in $\mathcal{X}_\mu(M)\setminus \mathcal{SR}_\mu(M)$, showing that $\mathcal{SR}_\mu(M)$ is not open.
\end{proof}
Notice that the homotopy class of the flow $X$ defined by $\iota_X\mu=d\alpha$ in the proof is arbitrary, as one can find contact structures in any homotopy class of plane fields. Thus, the non-openness of $\mathcal{SR}_\mu(M)$ in $\mathfrak{X}_\mu(M)$ holds in any homotopy class of non-vanishing fields.

\appendix

\section{Cohomologous Hamiltonian structures and embedded hypersurfaces}\label{s:app}

The following Lemma is probably known to the experts, but we add it here for completeness and future reference. It shows that small cohomologous perturbations of the Hamiltonian structure induced on an embedded hypersurface $M$ in a symplectic manifold $(W,\Omega)$ can be obtained as induced Hamiltonian structure on perturbations of $M$.

\begin{lemma}\label{lem:app}
Let $\varphi: M\rightarrow (W,\Omega)$ be a closed embedded hypersurface on a symplectic manifold, and consider the Hamiltonian structure $\omega=\varphi^*\Omega$. Then for any $\gamma \in \Omega^2(M)$ such that $[\gamma]=[\omega]$ in $H^2(M,\mathbb{R})$ and $\gamma$ is sufficiently $C^\infty$-close to $\omega$, there exists an embedding $\tilde \varphi:M\rightarrow (W,\Omega)$ that is $C^{\infty}$ close to $\varphi$ and such that $\tilde \varphi^*\Omega=\gamma$.
\end{lemma}

\begin{proof}
Choose some one-form $\lambda$ such that $\lambda\wedge \omega>0$ in $M$. The neighborhood theorem for hypersurfaces in symplectic manifolds says that there is a neighborhood of $\varphi(M)$ symplectomorphic to $U=M\times [-\varepsilon,\varepsilon]$ endowed with the symplectic form
$$ \Omega_0= \omega + d(t\lambda), $$
where $t$ denotes the coordinate in the second factor of $M\times [-\varepsilon,\varepsilon]$. That is, there exists a symplectic embedding $\psi: (U,\Omega_0) \rightarrow (M,\Omega)$ such that $\psi|_{M\times \{0\}}=\varphi$. 

Since $\gamma$ is $C^\infty$-close to $\omega$, the two-form $\gamma-\omega$ admits a $C^\infty$-small primitive $\eta$, this follows from example from \cite[Corollary 2.4.15]{Sch} and the Sobolev embedding theorem. Hence, given $\delta>0$, if $\gamma$ is sufficiently close to $\omega$ then we can assume that
$$ \norm{\eta}_{C^{\infty}}<\delta. $$
 Let $\phi:[-\varepsilon, \varepsilon]$ be a function constantly equal to $1$ near $0$ and equal to $0$ near $\pm \varepsilon$. Consider the family of two-forms
\begin{align*}
\Omega_s &= \Omega_0 + sd(\phi (t) \eta) \\
		&= \Omega_0 + s (\phi'(t)dt\wedge \eta + d\eta).
\end{align*}
The term $\phi'(t)$ is uniformly bounded in the $C^\infty$-topology (only depending on $\varepsilon$). Hence, given $\hat \delta>0$, there exists $\delta_0$ such that if $\delta<\delta_0$ the two-form $\Omega_s$ is non-degenerate for all $s$ and is $\hat \delta$-close to $\Omega_0$ in the $C^{\infty}$-topology. If $\iota: M \rightarrow M\times \{0\}\subset U$ denotes the inclusion of the zero section, observe that $\iota^*\Omega_0=\omega$. On the other hand $\iota^*\Omega_1=\omega + d\eta=\gamma$.\\

We have a path of cohomologous symplectic forms $\Omega_s$ such that $\Omega_s=\Omega_0$ for all $s$ on a neighborhood of $\partial U$. We are ready to apply Moser's path method. Let $X_s$ be the unique time-dependent family of vector fields defined by
\begin{equation}\label{eq:Mos}
\iota_{X_s}\Omega_s=-\phi(t)\eta.
\end{equation}
It follows that $X_s$, for a fixed value of the parameter, is necessarily $C^{\infty}$-small. Integrating $X_s$, we construct an isotopy $h_s: U \rightarrow U$ compactly supported in the interior of $U$ such that $h_s^*\Omega_t=\Omega_0$ and which is $C^{\infty}$-small. Let us denote by $g$ the inverse of $h_1$, which satisfies $g^*\Omega_0=\Omega_1$.

If $\iota:M \rightarrow M\times \{0\}\subset U$ denotes the inclusion of the zero section, we have $\varphi=\psi \circ \iota$ and hence $\varphi^*\Omega=\iota^*\psi^*\Omega=\iota^*\Omega_0=\omega$. Consider the embedding
$$ \tilde \varphi:= \psi \circ g \circ \iota: M\rightarrow (W,\Omega), $$
which satisfies $\tilde \varphi^*\Omega= \iota^*g^*\psi^*\Omega= \iota^*g^*\Omega_0=\iota^*\Omega_1= \gamma$ as we wanted. Notice that the embedding $\tilde \varphi$ is $C^{\infty}$-close to the original embedding.
\end{proof}

{

{\sc \noindent Robert Cardona}

\noindent Departament de Matem\`atiques i Inform\`atica, Universitat de Barcelona, Gran Via de Les Corts Catalanes 585, 08007 Barcelona, Spain. 

{\noindent  \em robert.cardona@ub.edu\/}
}

\end{document}